\documentclass[oneside,11pt,a4paper,reqno,ps]{article}
\usepackage[a4paper,top=2.5cm,bottom=2.5cm,left=2.5cm,right=2.5cm,
bindingoffset=5mm]{geometry}
\usepackage[latin1]{inputenc}
\usepackage[english]{babel}
\usepackage[T1]{fontenc}
\usepackage{syntonly}
\usepackage{graphicx}
\usepackage{amssymb}	
\usepackage{amsthm}
\usepackage{microtype}
\usepackage{amsmath}
\usepackage{braket}
\usepackage{eurosym}
\usepackage{amssymb}
\usepackage{amstext}
\usepackage{color}
\usepackage{url}
\usepackage{amsmath,amsthm,amssymb,amscd,epsfig,esint, mathrsfs,graphics,color}
 \usepackage{wrapfig}
\usepackage{verbatim}
\newcommand{\numberset}{\mathbb}
\newcommand{\R}{\numberset{R}}
\newcommand{\N}{\numberset{N}}
\newcommand{\Z}{\numberset{Z}}

\newcommand{\lnorm}{\left\Arrowvert}
\newcommand{\rnorm}{\right\Arrowvert}
\newcommand{\LBrace}{\left\lbrace}
\newcommand{\RBrace}{\right\rbrace}
\newcommand{\lAbs}{\left|}
\newcommand{\rAbs}{\right|}

\theoremstyle{plain}
\newtheorem{thm}{Theorem}[section]

\newtheorem{proposition}[thm]{Proposition}
\newtheorem{lemma}[thm]{Lemma}

\theoremstyle{definition}
\newtheorem{definition}[thm]{Definition}

\def\Xint#1{\mathchoice 
	{\XXint\displaystyle\textstyle{#1}}%
	{\XXint\textstyle\scriptstyle{#1}}%
	{\XXint\scriptstyle\scriptscriptstyle{#1}}%
	{\XXint\scriptscriptstyle\scriptscriptstyle{#1}}%
	\!\int} 
\def\XXint#1#2#3{{\setbox0=\hbox{$#1{#2#3}{\int}$} 
		\vcenter{\hbox{$#2#3$}}\kern-.5\wd0}} 
\def\Mint{\Xint -}

\def\div{{\rm div}}

\def\R{\mathbb{R}}
\numberwithin{equation}{section} \makeatletter

\renewcommand{\p@enumi}{\thesection.}
\title{\textbf{Higher differentiability results for solutions to a class of non-autonomous obstacle problems with sub-quadratic growth conditions}}
\author{Andrea Gentile}

\begin{document}

		\maketitle
			\begin{center}
				Università degli Studi di Napoli "Federico II", Dipartimento di Mat. e Appl. "R. Caccioppoli", Via Cintia, 80126 Napoli, Italy\\
				E-mail address: \url{andrea.gentile@unina.it}\\
				ORCID iD: \url{https://orcid.org/0000-0002-9830-6272}
			\end{center}

	\begin{abstract}
		\noindent We establish some higher differentiability results of integer and fractional order for solution to non-autonomous obstacle problems of the form
		
		\begin{equation*}
		\min \left\{\int_{\Omega}f(x, Dv(x))\,:\, v\in
		\mathcal{K}_\psi(\Omega)\right\},
		\end{equation*}
		
		\noindent where the function $f$ satisfies $p-$growth conditions with respect to the gradient variable, for $1<p<2$, and $\mathcal{K}_\psi(\Omega)$ is the class of admissible functions $v\in u_0+W^{1, p}_0(\Omega)$ such that $v\ge\psi$ a. e. in $\Omega$, where $u_0\in W^{1,p}(\Omega)$ is a fixed boundary datum.\\
		Here we show that a Sobolev or Besov-Lipschitz regularity assumption on the gradient of the obstacle $\psi$ transfers to the gradient of the solution, provided the partial map $x\mapsto D_\xi f(x,\xi)$ belongs to a suitable Sobolev or Besov space. The novelty here is that we deal with subquadratic growth conditions with respect to the gradient variable, i. e. $f(x, \xi)\approx a(x)|\xi|^p$ with $1<p<2,$ and where the map $a$ belongs to a Sobolev or Besov-Lipschitz space.
	\end{abstract}
	\noindent {\footnotesize{{\bf AMS Classifications.} 35J87; 
		49J40; 47J20.}}
	
	\bigskip
	\noindent {\footnotesize{{\bf Key words and phrases.} Obstacle problems; Higher differentiability; Sobolev coefficients; Besov-Lipschitz coefficients.}}
	\bigskip
	\section{Introduction}
	
	We are interested in the regularity properties of solutions to problems of the form
	
	\begin{equation}\label{functionalobstacle}
		\min \left\{\int_{\Omega}f(x, Dv(x))\,:\, v\in
		\mathcal{K}_\psi(\Omega)\right\},
	\end{equation}
	where $\Omega\subset\R^n$ is a bounded open set, $n>2$, $f:\Omega\times\R^n\to\R$ is a Carath\'{e}odory map, such that $\xi\mapsto f(x, \xi)$ is of class $C^2(\R^n)$ for a.e. $x\in\Omega$, $\psi: \Omega
	\mapsto [-\infty, +\infty)$ belonging to the Sobolev class $ W^{1,p}_{\mathrm{loc}}$ is the \emph{obstacle}, and
	
	$$\mathcal{K}_\psi(\Omega)=\left\lbrace v\in u_0+W^{1, p}_0(\Omega, \R): v\ge\psi \text{ a.e. in }\Omega\right\rbrace$$
	
	is the class of the admissible functions, with $u_0\in W^{1, p}(\Omega)$ a fixed boundary datum.
	
	\bigskip

	Let us observe that $u\in W^{1,p}_{\mathrm{loc}}(\Omega)$ is a
	solution to the obstacle problem \eqref{functionalobstacle} in
	$\mathcal{K}_\psi(\Omega)$ if and only if $u \in
	\mathcal{K}_\psi(\Omega)$ and $u$ is a solution to the variational
	inequality
	
	\begin{equation}\label{variationalinequality}
		\int_{\Omega}\left<A(x, Du(x)), D(\varphi(x)-u(x))\right>dx\ge0\qquad\forall
		\varphi\in \mathcal{K}_\psi(\Omega),
	\end{equation}
	where the operator $A: \Omega\times\R^n\to\R^n$ is defined
	as follows
	
	\begin{equation*}
		A_i(x, \xi)=D_{\xi_i}f(x, \xi)\qquad\forall i=1,...,n.
	\end{equation*}
	
	We assume that $A$ is a $p$-harmonic type operator, that is it satisfies the following $p$-ellipticity and $p$-growth conditions with respect to
	the $\xi$-variable. There exist positive constants $\nu, L, \ell$
	and an exponent $1<p\le2$ and a parameter $0\le \mu\le 1$  such that
	
	\begin{equation}\label{p-ellipticityA1}
	\left<A(x, \xi)-A(x, \eta), \xi-\eta\right>\ge\nu|\xi-\eta|^2\left(\mu^2+|\xi|^2+|\eta|^2\right)^\frac{p-2}{2},
	\end{equation}

	\begin{equation}\label{p-growtA2}
	\left|A(x, \xi)-A(x, \eta)\right|\le L|\xi-\eta|\left(\mu^2+|\xi|^2+|\eta|^2\right)^\frac{p-2}{2},
	\end{equation}
	
	\begin{equation}\label{p-growthA3}
	\left|A(x, \xi)\right|\le\ell\left(\mu^2+|\xi|^2\right)^\frac{p-1}{2},
	\end{equation}
	
	for all $\xi\in\R^n$ and for almost every $x\in\Omega.$\\
	
	The interest in the study of the regularity properties of solution to obstacle problems has been strongly increasing in the last decades as a research topic in Calulus of Variations and Partial Differential Equations.\\
	From the very beginning, obstacle problems were solved applying techniques of functional analysis, and it was clear soon that the regularity properties of the solutions were strictly connected to those of the obstacle.\\
	In the linear setting it was observed that the solutions and the obstacle have the same regularity (see \cite{BK, CaffarelliKinderlehrer, KS}), but it is not the same in the nonlinear case.\\
	Hence, along the years, there has been an intense research activity concerning the regularity properties of solutions to obstacle problems in the nonlinear setting (see \cite{ByunChoOk} and the references therein).\\		
	Many recent works deal with regularity properties of solutions to variational problems in which the integrand depends on the $x-$variable trough a function that is possibly discontinuous, such as in the case of Sobolev-type dependence, under quadratic (see \cite{PdN2}), and super-quadratic growth conditions (see \cite{Giova1, Giova2, Giova3, GiovaPassarelli, KristensenMingione,PdN1}).\\
	This kind of topics has been object of study also in the framework of obstacle problems (see \cite{CGG, EP1, EP2, MaZhang}), even in the case of $(p,q)$-growth condition (such as in \cite{CEP, Gavioli1, Gavioli2}).\\
	All the quoted papers show that the regularity of the obstacle influences the regularity of the solution, provided a suitable assumption is made on the map $x\mapsto f(x, \xi).$\\
	Already for unconstrained problems it is known that the sub-quadratic growth conditions require specific tools and, in general, the expected regularity of the solution, in the case $1<p<2$ strongly differs from the case $p\ge2$ (for a detailed explaination of this phenomenon see \cite{BDW}).\\
	We refer to the pioneering paper \cite{AcerbiFusco} in case of equations with H\"{o}lder-continuous coefficients (see also \cite{KM1, KM2, LeonettiMascoloSiepe}) and to \cite{Gentile1, Gentile2} for the case of Sobolev coefficients.\\
	The main aim of this paper is to extend to the sub-quadratic growth case some higher differentiability results for solutions to non-autonomous obstacle problems proved in \cite{EP2}.\\
	First, we show that an higher differentiability property of integer order of the gradient of the obstacle tranfers to the solution of problem \eqref{functionalobstacle}, provided the partial map $x\mapsto D_\xi f(x, \xi)$ belongs to a suitable Sobolev class, with no loss in the order of differentiation.\\
	More precisely  we assume that the map $x\mapsto A\left(x,\xi\right)$ belongs to $W^{1,n}_{\mathrm{loc}}(\Omega)$ for every $\xi\in\R^n$ or, equivalently, that there exists a non-negative function $g\in L^n_{\mathrm{loc}}(\Omega)$ such that
	
	\begin{equation}\label{x-dependence}
	\left|D_xA(x, \xi)\right|\le
	g(x)\left(\mu^2+|\xi|^2\right)^\frac{p-1}{2},
	\end{equation}
	(see \cite{Hajlasz}).\\
	
	Note that, since $f$, as a function of the $\xi$ variable, is of class $C^2$, then the operator $A$ is of class $C^1$ with respect to $\xi$, and \eqref{p-growtA2} implies
	
	\begin{equation}\label{A2bis}
	\left|D_\xi A(x, \xi)\right|\le c\left(\mu^2+|\xi|^2\right)^\frac{p-2}{2},
	\end{equation}
	
	for all $\xi\in\R^n\setminus\{0\}$ and for a. e. $x\in\Omega.$\\
	The first result we prove in this paper is the following.
	
	\begin{thm}\label{thm1}
		Let $u\in W^{1,p}_{\mathrm{loc}}(\Omega)$ be a solution to the obstacle problem \eqref{functionalobstacle} under assumptions \eqref{p-ellipticityA1}--\eqref{x-dependence} for $1<p<2$. Then the following implication holds:
		
		\begin{equation}\label{implication1}
		V_p\left(D\psi\right)\in W^{1,2}_{\mathrm{loc}}\left(\Omega\right)\Rightarrow V_p\left(Du\right)\in W^{1,2}_{\mathrm{loc}}\left(\Omega\right),
		\end{equation}
		
		with the following estimate
		
		\begin{align}\label{estimate1}
		\left\Arrowvert DV_p(Du(x))\right\Arrowvert_{L^2\left(B_{\frac{R}{2}}\right)}\le C\left(1+\left\Arrowvert Du\right\Arrowvert_{L^p(B_{2R})}+\left\Arrowvert V_p\left(D\psi\right)\right\Arrowvert_{W^{1, 2}(B_{2R})}+\left\Arrowvert g\right\Arrowvert_{L^n\left(B_{R}\right)}\right)^\sigma,
		\end{align}
		
		where $C$ and $\sigma$ are positive constants depending on $n, p, q, R, \alpha, \nu, L$ and $\ell$.
	\end{thm}
	
	Our next aim is to prove that the analogous phenomenon holds true in case the obstacle belongs to a Besov-Lipschitz space, provided we assume a Besov-Lipschitz dependence of the operator $A$ with respect to the $x$-variable. This represents, in some sense, the "fractional counterpart" of Theorem \ref{thm1}.\\
	More precisely, instead of \eqref{x-dependence}, we assume that, given $\alpha\in(0,1)$ and $1\le q<\infty$ there is a sequence of measurable non-negative functions $g_k\in L^{\frac{n}{\alpha}}(\Omega)$ such that
	$$
	\sum_k\Arrowvert g_k\Arrowvert^q_{L^{\frac{n}{\alpha}}(\Omega)}<\infty,
	$$ 
	and at the same time
	
	\begin{equation}\label{A5}
	\left|A(x, \xi)-A(y, \xi)\right|\le\left(g_k(x)+g_k(y)\right)|x-y|^\alpha\left(\mu^2+|\xi|^2\right)^\frac{p-1}{2},
	\end{equation}
	
	for each $\xi\in\R^n$ and almost every $x, y\in\Omega$ such that $2^{-k}\mathrm{diam}(\Omega)\le|x-y|\le2^{-k+1}\mathrm{diam}(\Omega)$. We will shortly write then that $\left(g_k\right)_k\in\ell^q\left(L^\frac{n}{\alpha}(\Omega)\right).$
	If $A(x, \xi)=\gamma(x)|\xi|^{p-2}\xi$ and $\Omega=\R^n$ then \eqref{A5} says that $\gamma\in B^\alpha_{\frac{n}{\alpha},q}.$\\
	
	It is worth noticing that, due to the sub-quadratic growth conditions, the Besov regularity of the obstacle transfers to the solution with a small loss in the order of differentiations.
	
	\begin{thm}\label{thm2}
		Let $u\in W^{1, p}_{\mathrm{loc}}(\Omega)$ be a solution to the obstacle problem \eqref{functionalobstacle}, under the assumptions \eqref{p-ellipticityA1}--\eqref{p-growthA3} and \eqref{A5}, for $1<p<2$. Then the following implication holds 
		
		\begin{equation}\label{implication2}
		V_p\left(D\psi\right)\in B^{\alpha}_{2,q,\mathrm{loc}}(\Omega)\Rightarrow V_p\left(Du\right)\in B^{\alpha\beta}_{2,q,\mathrm{loc}}(\Omega)
		\end{equation}
		
		for any $q\le 2^*_\alpha=\frac{2n}{n-2\alpha}$ and $\beta\in(0, 1)$.\\
		Moreover, for any ball $B_{4R}\Subset\Omega$, the following estimate holds
		
		\begin{align}\label{estimate2}
		\left\Arrowvert\frac{\tau_h V_p\left(Du\right)}{|h|^{\alpha\beta}}\right\Arrowvert_{L^q\left(\frac{dh}{|h|^n}; L^2\left(B_{\frac{R}{2}}\right)\right)}\le C&\Bigg(1+\left\Arrowvert Du\right\Arrowvert_{L^p(B_{4R})}+\left\Arrowvert V_p\left(D\psi\right)\right\Arrowvert_{B^\alpha_{2,q}(B_{4R})}\notag\\
		&+\left\Arrowvert\left\lbrace g_k\right\rbrace_k\right\Arrowvert_{\ell^q\left(L^\frac{n}{\alpha}\left(B_{2R}\right)\right)}\Bigg)^\sigma,
		\end{align}
				
		where $C$ and $\sigma$ are positive constants depending on $n, p, q, R, \alpha, \nu, L$ and $\ell$.
	\end{thm}

	In the Besov-Lipschitz framework, if $q=\infty$, we still have that a fractional differentiability property of the obstacle transfers to the solution with a larger loss on the order of differentiation than the one we have when $q$ is finite. This is due to the fact that the regularity of the type $B^\alpha_{p,\infty}$ is the weakest one to assume both on the coefficients and on the gradient of the obstacle (see Lemmas \ref{Lemma2.9EP} and \ref{EP2.3} in Section \ref{Preliminaries} below).\\
	More precisely, we prove the following.
	
	\begin{thm}\label{thm3}
		Let $u\in W^{1, p}_{\mathrm{loc}}(\Omega)$ be a solution to the obstacle problem \eqref{functionalobstacle}, under the assumptions \eqref{p-ellipticityA1}--\eqref{p-growthA3} for $1<p<2$. If there exists $\alpha\in(0,1)$ and a function $g\in L^\frac{n}{\alpha}_{\mathrm{loc}}(\Omega)$ such that
		
		\begin{equation}\label{A6}
		\left|A(x, \xi)-A(y, \xi)\right|\le\left(g(x)+g(y)\right)|x-y|^\alpha\left(\mu^2+|\xi|^2\right)^\frac{p-1}{2},
		\end{equation}
		
		for a. e. $x, y\in\Omega$ and for every $\xi\in\R^n$, then, provided $0<\alpha<\gamma<1$ the following implication holds
		
		\begin{equation}\label{implication3}
		V_p\left(D\psi\right)\in B^{\gamma}_{2,\infty,\mathrm{loc}}(\Omega)\Rightarrow V_p\left(Du\right)\in B^{\alpha\beta}_{2,\infty,\mathrm{loc}}(\Omega),
		\end{equation}
		for any $\beta\in(0, 1).$\\
		Moreover, for any ball $B_{4R}\Subset\Omega$, the following estimate holds
		
		\begin{align}\label{estimate3}
		\left[V_p\left(Du\right)\right]_{\dot{B}^{\alpha\beta}_{2,\infty}\left(B_{\frac{R}{2}}\right)}\le& C\left(1+\left\Arrowvert Du\right\Arrowvert_{L^p(B_{4R})}+\left\Arrowvert V_p\left(D\psi\right)\right\Arrowvert_{B^\gamma_{2,\infty}(B_{4R})}\right.\notag\\
		&\left.\quad+\left\Arrowvert g\right\Arrowvert_{L^\frac{n}{\alpha}\left(B_{2R}\right)}\right)^\sigma,
		\end{align}
		
		where $C$ and $\sigma$ are positive constants depending on $n, p, q, R, \alpha, \beta,\gamma, \nu, L$ and $\ell$.
	\end{thm}
	The main difference between the Sobolev and the Besov setting is due to the fact that, in the Sobolev case, we can use a very well known linearization technique based on the fact that solving the problem is equivalent to solve an equation whose right-hand side is different from zero only in the set where the solution coincides with the obstacle (see \cite{Fuchs, FuchsMingione}). Here we take advantage from this method thanks to the Sobolev regularity of the gradient of the obstacle.\\
	Differently, in the Besov case, we need to start from the variational inequality, since we can't exploit the calculations in the right-hand side of the equation that comes from the linearization technique (see \eqref{diveq} below).\\
	In both cases, essential tools are the difference-quotient method and Calderòn-Zygmund type estimates proved in \cite{ByunChoOk}. We take advantage from this Calderòn-Zygmund estimates, since our assumptions on the map $x\mapsto A(x,\xi)$ imply its $VMO$ regularity (see Lemma \ref{lemma2.11} below).\\
	
	We conclude this introduction with a brief description of the structure of this paper. Section \ref{Preliminaries} is devoted to the preliminaries: after a list of some classical notations, and some general results, we recall some classical properties of difference quotients of Sobolev functions, basic properties of Besov-Lipschitz spaces and, at last, some useful results involving variational problems with $VMO$ coefficients. In Section \ref{Thm1Pf} we prove Theorem \ref{thm1}, in Section \ref{Thm2Pf} the proof of Theorem \ref{thm2} is given, and the paper concludes with the proof of Theorem \ref{thm3} in Section \ref{Thm3Pf}.
	
	\section{Notations and preliminary results}\label{Preliminaries}
	
	In this section we list the notations that we use in this paper and recall some tools that will be useful to prove our results.\\
	We shall follow the usual convention and denote by $C$ or $c$ a general constant that may vary on different occasions, even within the same
	line of estimates. Relevant dependencies on parameters and special
	constants will be suitably emphasized using parentheses or
	subscripts. The norm we use on $\R^n$, will be the standard Euclidean one.\\
	For  a  $C^2$ function $f \colon \Omega\times\R^{n} \to \R$, we write
	$$
	D_\xi f(x,\xi )[\eta ] := \frac{\rm d}{{\rm d}t}\Big|_{t=0} f(x,\xi
	+t\eta )\quad \mbox{ and } \quad D_{\xi\xi}f(x,\xi )[\eta ,\eta ] :=
	\frac{\rm d^2}{{\rm d}t^{2}}\Big|_{t=0} f(x,\xi +t\eta )
	$$
	for $\xi$, $\eta \in \R^{n}$ and for almost every $x\in \Omega$.\\
	With the symbol $B(x,r)=B_r(x)=\{y\in
	\R^n:\,\, |y-x|<r\}$, we will denote the ball centered at $x$ of
	radius $r$ and
	$$(u)_{x_0,r}= \Mint_{B_r(x_0)}u(x)\,dx,$$
	stands for the integral mean of $u$ over the ball $B_r(x_0)$. We
	shall omit the dependence on the center  when it is clear from the context.
	In the following, we will denote, for any ball
	$B=B_r(x_0)=\{x\in\R^n: |x-x_0|<r\}\Subset\Omega$
	
	\begin{equation}
		\Mint_Bu(x)dx=\frac{1}{|B|}\int_Bu(x)dx.
	\end{equation}
	
	Here we recall some results that will be useful in the following.
	We will use the auxiliary function $V_p:\R^n\to\R^n$, defined as 
	
	\begin{equation}\label{Vp}
		V_p(\xi):=\left(\mu^2+|\xi|^2\right)^\frac{p-2}{4}\xi,
	\end{equation}
	
	\noindent for which the following estimates hold (see \cite{GM} for the case $p\ge2$ and \cite{AcerbiFusco} for the case $1<p<2$).
	
	\begin{lemma}\label{lemma6GP}
		Let $1<p<\infty$. There is a constant $c=c(n, p)>0$ such that
		
		\begin{equation}\label{lemma6GPestimate1}
			c^{-1}\left(\mu^2+|\xi|^2+|\eta|^2\right)^\frac{p-2}{2}\le\frac{\left|V
				_p(\xi)-V_p(\eta)\right|^2}{|\xi-\eta|^2}\le c\left(\mu^2+|\xi|^2+|\eta|^2\right)^\frac{p-2}{2},
		\end{equation}
		
		\noindent for any $\xi, \eta\in\R^n.$
		Moreover, for a $C^2$ function $g$, there is a constant $C(p)$ such that
		
		\begin{equation}\label{lemma6GPestimate2}
			C^{-1}\left|D^2g\right|^2\left(\mu^2+\left|Dg\right|^2\right)^\frac{p-2}{2}\le\left|DV_p(Dg)\right|^2\le C\left|D^2g\right|^2\left(\mu^2+\left|Dg\right|^2\right)^\frac{p-2}{2}
		\end{equation}.
	\end{lemma}
	
	Let us conclude this section with a result (see \cite{Fuchs, FuchsMingione}), that is useful to prove Theorem \ref{thm1}.
	
	\begin{thm}\label{Fuchs}
		A function $u\in W^{1, p}_{\mathrm{loc}}(\Omega)$ is a solution to the problem \eqref{functionalobstacle} if and only if it is a weak solution of the following equation:
		
	\begin{equation}\label{diveq}
	\div A\left(x, Du(x)\right)=-\div A(x, D\psi(x))\chi_{\set{u=\psi}}(x).
	\end{equation}
	\end{thm}

\subsection{Difference quotients}
	\medskip
	In order to get the regularity of the solutions
	to problem \eqref{functionalobstacle}, we shall use the
	difference quotient method. We recall here the definition
	and basic results.
	\begin{definition}
		Given $h\in\mathbb{R}$, for every function
		$F:\mathbb{R}^{n}\to\mathbb{R}$ the finite difference operator is
		defined by
		$$
		\tau_{h}F(x)=F(x+h)-F(x).
		$$
	\end{definition}
	\par
	We recall some properties of the   finite difference operator that
	will be needed in the sequel. We start with the description of some
	elementary properties that can be found, for example, in
	\cite{Giusti}.
	
	\bigskip

	\begin{proposition}\label{findiffpr}
		
		Let $F$ and $G$ be two functions such that $F, G\in
		W^{1,p}(\Omega)$, with $p\geq 1$, and let us consider the set
		$$
		\Omega_{|h|}:=\left\{x\in \Omega : \mathrm{dist}(x,
		\partial\Omega)>|h|\right\}.
		$$
		Then
		\begin{description}
			\item{$(a)$} $\tau_{h}F\in W^{1,p}(\Omega_{|h|})$ and
			$$
			D_{i} (\tau_{h}F)=\tau_{h}(D_{i}F).
			$$
			\item{$(b)$} If at least one of the functions $F$ or $G$ has support contained
			in $\Omega_{|h|}$ then
			$$
			\int_{\Omega} F(x)\, \tau_{h} G(x)\, dx =\int_{\Omega} G(x)\, \tau_{-h}F(x)\,
			dx.
			$$
			\item{$(c)$} We have
			$$
			\tau_{h}(F G)(x)=F(x+h )\tau_{h}G(x)+G(x)\tau_{h}F(x).
			$$
		\end{description}
	\end{proposition}

	\noindent The next result about finite difference operator is a kind
	of integral version of Lagrange Theorem.
	\begin{lemma}\label{le1} If $0<\rho<R$, $|h|<\frac{R-\rho}{2}$, $1 < p <+\infty$,
		and $F, DF\in L^{p}(B_{R})$ then
		$$
		\int_{B_{\rho}} |\tau_{h} F(x)|^{p}\ dx\leq c(n,p)|h|^{p}
		\int_{B_{R}} |D F(x)|^{p}\ dx .
		$$
		Moreover
		$$
		\int_{B_{\rho}} |F(x+h )|^{p}\ dx\leq  \int_{B_{R}} |F(x)|^{p}\ dx .
		$$
	\end{lemma}
	
	We also need to recall this result, that is proved in \cite{Giusti}.
	
	\begin{lemma}\label{Giusti8.2}
		Let $F:\R^n\to\R^N$, $F\in L^p(B_R)$ with $1<p<+\infty$. Suppose that there exist $\rho\in(0, R)$ and $M>0$ such that
		
		$$
		\sum_{s=1}^{n}\int_{B_\rho}|\tau_{s, h}F(x)|^pdx\le M^p|h|^p
		$$
		
		for every $h<\frac{R-\rho}{s}$. Then $F\in W^{1,p}(B_R, \R^N)$. Moreover
		
		$$
		\Arrowvert DF \Arrowvert_{L^p(B_\rho)}\le M.
		$$
			
		and	
		
		$$
		\left\Arrowvert F\right\Arrowvert_{L^{\frac{np}{n-p}}(B\rho)}\le c\left(M+\left\Arrowvert F\right\Arrowvert_{L^p(B_R)}\right),
		$$
		
		with $c=c(n, N, p, \rho, R).$
	\end{lemma}

Before introducing Besov-Lipschitz spaces, we conclude this section recalling a fractional version of Lemma \ref{Giusti8.2}, whose proof can be found in \cite{KristensenMingione}.

\begin{lemma}\label{Lemma2.9EP}
	Let $F\in L^2(B_R)$. Suppose that there exist $\rho\in(0, R)$, $\alpha\in(0, 1)$ and $M>0$ such that
	
	$$
	\sum_{s=1}^{n}\int_{B_\rho}|\tau_{s, h}F(x)|^2dx\le M^2|h|^{2\alpha},
	$$
	
	for every $h<\frac{R-\rho}{s}$. Then $F\in L^{\frac{2n}{n-2\beta}}(B_\rho)$ for every $\beta\in(0, \alpha)$ and	
	
	$$
	\left\Arrowvert F\right\Arrowvert_{L^{\frac{2n}{n-2\beta}}(B\rho)}\le c\left(M+\left\Arrowvert F\right\Arrowvert_{L^2(B_R)}\right),
	$$
	
	with $c=c(n, N, p, \rho, R, \alpha, \beta).$
\end{lemma}

\subsection{Besov-Lipschitz spaces}

Let us consider $0<\alpha<1$ and $1\le p,q<\infty$ and, for a function $v:\R^n\to\R$ and $h\in\R^n$, we denote, like in the previous section, $\tau_{h}v(x)=v(x+h)-v(x)$. We say that $v$ belongs to the Besov-Lischitz space $B^\alpha_{p,q}\left(\R^n\right)$ if $v\in L^p(\R^n)$ and 

\begin{equation}\label{B-Lsemiorm}
[v]_{\dot{B}^\alpha_{p,q}}=\left(\int_{\R^n}\left(\int_{\R^n}\frac{\left|\tau_{h}v(x)\right|^p}{|h|^{\alpha p}}dx\right)^\frac{q}{p}\frac{dh}{|h|^n}\right)^\frac{1}{q}<\infty.
\end{equation}

We define a norm in the space $B^\alpha_{p,q}\left(\R^n\right)$ as follows

\begin{equation}\label{Besovpnorm}
\Arrowvert v\Arrowvert_{B^\alpha_{p,q}\left(\R^n\right)}=\Arrowvert v\Arrowvert_{L^p(\R^n)}+[v]_{\dot{B}^\alpha_{p,q}},
\end{equation}

and with this norm $B^\alpha_{p,q}\left(\R^n\right)$ is a Banach space.\\
Equivalently, we could say that a function $v\in L^p(\R^n)$ belongs to $B^\alpha_{p,q}(\R^n)$ if and only if $\frac{\tau_hv}{|h|^\alpha}\in L^q\left(\frac{dh}{|h|^n}; L^p(\R^n)\right)$. We can also observe that, in \eqref{B-Lsemiorm}, one can simply integrate for $h\in B(0, \delta)$ for a fixed $\delta>0$, thus obtaining an equivalent norm, because

$$
\left(\int_{\left\lbrace|h|\ge\delta\right\rbrace}\left(\int_{\R^n}\frac{\left|\tau_{h}v(x)\right|^p}{|h|^{\alpha p}}dx\right)^\frac{q}{p}\frac{dh}{|h|^n}\right)^\frac{1}{q}\le c(n, \alpha, p, q, \delta)\left\Arrowvert v\right\Arrowvert_{L^p(\R^n)}.
$$
\\
Moreover, for a function $v\in L^p(\R^n)$, we say that $v\in B^{\alpha}_{p,\infty}(\R^n)$ if

\begin{equation}\label{B-LInfsemiorm}
[v]_{\dot{B}^\alpha_{p,\infty}}=\sup_{h\in\R^n}\left(\int_{\R^n}\frac{\left|\tau_{h}v(x)\right|^p}{|h|^{\alpha p}}dx\right)^\frac{1}{p}<\infty,
\end{equation}

and we define the following norm

\begin{equation}\label{BesovInfnorm}
\Arrowvert v\Arrowvert_{B^\alpha_{\infty,q}\left(\R^n\right)}=\Arrowvert v\Arrowvert_{L^\infty(\R^n)}+[v]_{\dot{B}^\alpha_{\infty,q}}.
\end{equation}

Again, in \eqref{B-LInfsemiorm}, the supremum can be taken over the set $\left\lbrace|h|\le\delta\right\rbrace$ for a fixed $\delta>0$, and the norm that we obtain is equivalent.\\ 
By construction, $B^\alpha_{p,q}(\R^n)\subset L^p(\R^n)$. Moreover, the following Sobolev-type embeddings hold for Besov-Lipschitz spaces.

\begin{lemma}\label{EP2.2}
	Suppose that $0<\alpha<1$.
	\begin{description}
		\item{$(a)$} If $1<p<\frac{n}{\alpha}$ and $1\le q\le p^*_\alpha=\frac{np}{n-\alpha p}$, then there is a continuous embedding $B^\alpha_{p,q}(\R^n)\subset L^{p^*_\alpha}(\R^n).$
		\item{$(b)$} If $p=\frac{n}{\alpha}$ and $1\le q\le \infty,$ then there is a continuous embedding $B^\alpha_{p,q}(\R^n)\subset BMO(\R^n)$,
	\end{description}
	 	where BMO denotes the space of functions with bounded mean oscillations.
\end{lemma}

The following lemma describes the inclusions between Besov-Lipschitz spaces.

\begin{lemma}\label{EP2.3}
	Suppose that $0<\beta<\alpha<1$.
	\begin{description}
		\item{$(a)$} If $1<p<\infty$ and $1\le q\le r\le\infty$ then $B^\alpha_{p,q}(\R^n)\subset B^\alpha_{p,r}(\R^n).$
		\item{$(b)$} If $1<p<\infty$ and $1\le q,r\le\infty$ then $B^\alpha_{p,q}(\R^n)\subset B^\beta_{p,r}(\R^n).$
		\item{$(c)$} If $1\le q\le \infty$, then $B^\alpha_{\frac{n}{\alpha},q}(\R^n)\subset B^\beta_{\frac{n}{\beta},q}(\R^n).$
	\end{description}
\end{lemma}

For the proofs of Lemmas \ref{EP2.2} and \ref{EP2.3} we refer to \cite{Hasroke}.
We can also define local Besov-Lipschitz spaces as follows. 
Let $\Omega\subset\R^n$ be a bounded open set. We say that a function $v$ belongs to $B^\alpha_{p,q,\mathrm{loc}}(\Omega)$ if, for any smooth function with compact support in $\Omega$, $\varphi\in C^\infty_0(\Omega)$, we have $\varphi v\in B^\alpha_{p,q}(\R^n)$.
It is easy to extend the embeddings described in Lemma \ref{EP2.2} and \ref{EP2.3} even to local Besov spaces.
The following Lemma is an easy consequence of the definitions given above and its proof can be found in \cite{BCGOP}.

\begin{lemma}\label{lemma2.4EP}
	A function $v\in L^p_{\mathrm{loc}}(\Omega)$ belongs to the local Besov space $B^\alpha_{p,q,\mathrm{loc}}(\Omega)$ if and only if 
	
	\begin{equation*}
	\left\Arrowvert\frac{\tau_hv}{|h|^\alpha}\right\Arrowvert_{L^q\left(\frac{dh}{|h|^n}; L^p(B)\right)}<\infty
	\end{equation*}
	
	for any ball $B\subset2B\subset\Omega$ with radius $r_B$. Here the measure $\frac{dh}{|h|^n}$ is restricted to the ball $B(0, r_B)$ on the $h$-space.
\end{lemma}

It is known that Besov-Lipschitz spaces of fractional order $\alpha\in(0, 1)$ can be characterized
in pointwise terms. Given a measurable function $v: \R^n\to\R$, a fractional $\alpha$-Haj\l asz gradient for $v$ is a sequence $(g_k)_k$ of measurable, non-negative functions $g_k: \R^n\to\R$, together with
a null set $N\subset\R^n$ such that the inequality

\begin{equation*}
|v(x)-v(y)|\le\left(g_k(x)+g_k(y)\right)|x-y|^\alpha,
\end{equation*}

holds for any $k\in\Z$ and $x, y\in\R^n\setminus N$ are such that $2^{-k}\le|x-k|\le2^{-k+1}$. We say that $(g_k)\in\ell^q\left(\Z; L^p(\R^n)\right)$ if 

\begin{equation*}
\left\Arrowvert(g_k)_k\right\Arrowvert_{\ell^q(L^p)}=\left(\sum_{k\in\Z}\left\Arrowvert g_k\right\Arrowvert^q_{L^p(\R^n)}\right)^\frac{1}{q}<\infty.
\end{equation*}

The following result is proved in \cite{KYZ}.

\begin{thm}\label{Thm2.5EP}
	Let $\alpha\in(0, 1)$, $1\le p<\infty$ and $1\le q\le\infty.$ Let $v\in L^p(\R^n).$ One has $v\in B^\alpha_{p,q}(\R^n)$ if and only if there exists a fractional $\alpha$-Haj\l asz gradient $(g_k)_k\in\ell^q(\Z; L^p(\R^n))$ for $v$. Moreover,
	
	\begin{equation*}
	\left\Arrowvert v\right\Arrowvert_{B^\alpha_{p,q}\left(\R^n\right)}\simeq\inf\left\Arrowvert(g_k)_k\right\Arrowvert_{\ell^q(L^p)},
	\end{equation*}
	
	where the infimum runs over all the possible $\alpha$-Haj\l asz gradients for $v$.
\end{thm}

For further needs, we record the following.

\begin{lemma}\label{lemmapreliminare}
	Let $\Omega\subset\R^n$ a bounded open set, $1<p<2$, $\alpha\in(0,1)$ and $1\le q\le\infty.$
	Then the following implication holds
	
	\begin{equation}\label{implicazione}
	V_p\left(D\psi\right)\in B^{\alpha}_{2,q,\mathrm{loc}}(\Omega)\Rightarrow D\psi\in B^\alpha_{p,q,\mathrm{loc}}(\Omega).
	\end{equation}
	
	Moreover, for any ball $B_R\Subset\Omega$ and $0<\rho<R$, the following estimate
	
	\begin{equation}\label{preliminarestimate}
	\left[D\psi\right]_{\dot{B}^\alpha_{p, q}(B_\rho)}\le C\left(1+\left\Arrowvert D\psi\right\Arrowvert_{L^p\left(B(R)\right)}+\left\Arrowvert V_p\left(D\psi\right)\right\Arrowvert_{B^\alpha_{2, q}(B_{R})}\right)^\sigma
	\end{equation}
	
	holds true for $1\le q\le\infty$, where $C$ and $\sigma$ are positive constants depending on $n, p, \alpha$ and $q.$
\end{lemma}

\begin{proof}
Let us fix a ball $B_{R}(x_0)\Subset\Omega$ and $0<\rho<R.$\\
Since $V_p\left(D\psi\right)\in B^{\alpha}_{2,q,\mathrm{loc}}(\Omega)$, then, by definition, $V_p(D\psi)\in L^2_{\mathrm{loc}}(\Omega)$, and so it's easy to check that $D\psi\in L^p_{\mathrm{loc}}(\Omega).$\\
More precisely, using H\"{o}lder's Inequality with exponents $\left(\frac{2}{p}, \frac{2}{2-p}\right)$ and Young's Inequality with the same exponents, we have

\begin{align}
\int_{B_{R}}\left|D\psi(x)\right|^pdx=&\int_{B_{R}}\left|D\psi(x)\right|^p\left(\mu^2+\left|D\psi(x)\right|^2\right)^\frac{p(p-2)}{4}\cdot\left(\mu^2+\left|D\psi(x)\right|^2\right)^\frac{p(2-p)}{4}dx\notag\\
\le&\left(\int_{B_{R}}\left|D\psi(x)\right|^2\left(\mu^2+\left|D\psi(x)\right|^2\right)^\frac{p-2}{2}dx\right)^\frac{p}{2}\cdot\left(\int_{B_{R}}\left(\mu^2+\left|D\psi(x)\right|^2\right)^\frac{p}{2}dx\right)^\frac{2-p}{2}\notag\\
\le& c_p\left(\int_{B_{R}}\left|D\psi(x)\right|^2\left(\mu^2+\left|D\psi(x)\right|^2\right)^\frac{p-2}{2}dx\right)^\frac{p}{2}\cdot\left(\int_{B_{R}}\left(\mu^p+\left|D\psi(x)\right|^p\right)dx\right)^\frac{2-p}{2}\notag\\
\le&c_\varepsilon\int_{B_{R}}\left|D\psi(x)\right|^2\left(\mu^2+\left|D\psi(x)\right|^2\right)^\frac{(p-2)}{2}dx+\varepsilon\int_{B_{R}}\left(\mu^p+\left|D\psi(x)\right|^p\right)dx,
\end{align}

an so, choosing $\varepsilon$ sufficiently small and recalling the definition of $V_p$, given at \eqref{Vp}, we get

\begin{align}
\int_{B_{R}}\left|D\psi(x)\right|^pdx\le&C\left(\int_{B_{R}}\left|V_p\left(D\psi(x)\right)\right|^2dx+1\right),
\end{align}

where the positive constant $C$ depends on $n$ and $p$.\\
Now, let us consider, first, the case $1\le q<\infty$.\\
Using H\"{o}lder's Inequality with exponents $\left(\frac{2}{p}, \frac{2}{2-p}\right)$, Lemmas \ref{lemma6GP} and \ref{le1}, we have

\begin{align}\label{finiteq}
&\int_{B_{\frac{R}{2}}(0)}\left(\int_{B_{\rho}}\frac{\left|\tau_{h}D\psi(x)\right|^p}{|h|^{p\alpha}}dx\right)^\frac{q}{p}\frac{dh}{|h|^n}\notag\\
=&\int_{B_{\frac{R}{2}}(0)}\left[\left(\int_{B_{\rho}}\frac{\left|\tau_{h}D\psi(x)\right|^p}{|h|^{p\alpha}}\right)\cdot\left(\mu^2+\left|D\psi(x)\right|^2+\left|D\psi(x+h)\right|^2\right)^{\frac{p(p-2)}{4}}\right.\notag\\
&\left.\quad\cdot\left(\mu^2+\left|D\psi(x)\right|^2+\left|D\psi(x+h)\right|^2\right)^{\frac{p(2-p)}{4}}dx\right]^{\frac{q}{p}}\frac{dh}{|h|^n}\notag\\
\le&\int_{B_{\frac{R}{2}}(0)}\left[\int_{B_{\rho}}\frac{\left|\tau_{h}D\psi(x)\right|^2}{|h|^{2\alpha}}\cdot\left(\mu^2+\left|D\psi(x)\right|^2+\left|D\psi(x+h)\right|^2\right)^{\frac{(p-2)}{2}}dx\right]^\frac{q}{2}\notag\\
&\quad\cdot\left[\int_{B_\rho}\left(\mu^2+\left|D\psi(x)\right|^2+\left|D\psi(x+h)\right|^2\right)^{\frac{p}{2}}dx\right]^{\frac{2-p}{2}\cdot\frac{q}{p}}\frac{dh}{|h|^n}\notag\\
\le& c\left[\int_{B_{R}}\left(\mu^2+\left|D\psi(x)\right|^2\right)^\frac{p}{2}dx\right]^{\frac{2-p}{2}\cdot\frac{q}{p}}\cdot\left[\int_{B_{\frac{R}{2}}(0)}\left(\int_{B_{\rho}}\frac{\left|\tau_{h}V_p\left(D\psi(x)\right)\right|^2}{|h|^{2\alpha}}dx\right)^{\frac{q}{2}}\frac{dh}{|h|^n}\right],
\end{align}

and the right-hand side of \eqref{finiteq} is finite since, as we proved above, $D\psi\in L^p_{\mathrm{loc}}(\Omega)$, and $V_p(D\psi)\in B^{\alpha}_{2, q, \mathrm{loc}}(\Omega)$ by hypothesis.\\
Let us consider, now, the case $q=\infty$. Arguing as above, we have, 

\begin{align}\label{infq}
\left(\int_{B_{\rho}}\frac{\left|\tau_{h}D\psi(x)\right|^p}{|h|^{p\alpha}}dx\right)^\frac{1}{p}\le& c\left[\int_{B_{R}}\left(\mu^2+\left|D\psi(x)\right|^2\right)^\frac{p}{2}dx\right]^{\frac{2-p}{2}\cdot\frac{1}{p}}\notag\\
&\quad\cdot\left(\int_{B_{\rho}}\frac{\left|\tau_{h}V_p\left(D\psi(x)\right)\right|^2}{|h|^{2\alpha}}dx\right)^{\frac{1}{2}},
\end{align}

and taking the supremum for $|h|<\frac{R}{2}$, since, by hypothesis, $V_p(D\psi)\in B^{\alpha}_{2, \infty, \mathrm{loc}}(\Omega)$, we have $D\psi\in B^{\alpha}_{p, \infty, \mathrm{loc}}(\Omega)$. 

Recalling the definition of the norms in Besov-Lipschitz spaces, and applying Young's Inequality to \eqref{finiteq} and \eqref{infq}, for a suitable choice of $C$ and $\sigma$, we conclude with the following estimate

\begin{equation}\label{preliminarestimatepf}
\left[D\psi\right]_{\dot{B}^\alpha_{p, q}(B_\rho)}\le C\left(1+\left\Arrowvert D\psi\right\Arrowvert_{L^p\left(B(R)\right)}+\left\Arrowvert V_p\left(D\psi\right)\right\Arrowvert_{B^\alpha_{2, q}(B_{R})}\right)^\sigma
\end{equation}

holding true for $1\le q\le\infty.$
\end{proof}

\subsection{${\bf VMO}$ coefficients}

In order to prove our results, we shall use the fact that, if the operator $A$ satisfies \eqref{p-ellipticityA1}, \eqref{p-growtA2}, \eqref{p-growthA3} and \eqref{x-dependence} or \eqref{A5}, then it is locally uniformly in $VMO$ (see \cite{CGP}). More precisely, given a ball $B\subset\Omega$, let us introduce the operator

\begin{equation*}
A_B=\Mint_BA(x, \xi).
\end{equation*}

One can easily check that $A_B(\xi)$ also satisfies \eqref{p-ellipticityA1}, \eqref{p-growtA2} and \eqref{p-growthA3}. Setting

\begin{equation}\label{VMOOperator}
V(x, B)=\sup_{\xi\neq 0}\frac{\left|A(x, \xi)-A_B(\xi)\right|}{\left(\mu^2+\left|\xi\right|^2\right)^\frac{p-1}{2}},
\end{equation}

we will say that $x\mapsto A(x, \xi)$ is locally uniformly in $VMO$ if for each compact set $K\subset\Omega$ we have that 

\begin{equation}\label{VMOOperator1}
\lim_{R\to 0}\sup_{r<R}\sup_{x_0\in K}\Mint_{B_r(x_0)}V(x, B)dx=0.
\end{equation}

Next Lemma will be a key tool to prove our results. Its proof, for $p\ge2$, can be found in \cite{CGP} (Lemma 3.1), but i holds, exactly in the same way also for $1<p<2$.

\begin{lemma}\label{lemma2.11}
	Let $A$ be such that \eqref{p-ellipticityA1}, \eqref{p-growtA2}, \eqref{p-growthA3} and \eqref{x-dependence} or \eqref{A5} hold. Then $A$ is locally uniformly in $VMO$, that is \eqref{VMOOperator1} holds.
\end{lemma}

The following Theorem is a Calderón-Zygmund type estimate for solutions to the obstacle problem with $VMO$ coefficients, and its proof can be found in \cite{ByunChoOk} (in the case $p=p(x)$).

\begin{thm}\label{ThmByunChoOk}
	Let $p>1$, and $q>p$. Assume that \eqref{p-ellipticityA1}, \eqref{p-growtA2}, \eqref{p-growthA3} hold, and that
	$x \mapsto A(x, \xi)$ is locally uniformly in $VMO$ Let $u\in\mathcal{K}_{\psi}(\Omega)$ be the solution to the obstacle
	problem \eqref{functionalobstacle}. Then the following implication holds
	
	\begin{equation}\label{ByunChoOkImplicazione}
	D\psi\in L^q_{\mathrm{loc}}(\Omega)\Rightarrow Du\in L^q_{\mathrm{loc}}(\Omega).
	\end{equation}
	
	Moreover, there exists a constant $C=C(n, \nu, \ell, L, p, q)$ such that the following inequality
	
	\begin{equation}\label{ByunChoOkStima}
	\Mint_{B_{R}}\left|Du(x)\right|^qdx\le C\left\lbrace1+\Mint_{B_{2R}}\left|D\psi(x)\right|^qdx+\left(\Mint_{B_{2R}}\left|Du(x)\right|^pdx\right)^\frac{q}{p}\right\rbrace
	\end{equation}
	
	holds for any ball $B_R$ such that $B_{2R}$.
\end{thm}

\section{Proof of Theorem \ref{thm1}}\label{Thm1Pf}
	
	\begin{proof}
	In order to apply Theorem \ref{Fuchs}, let us recall that $u\in W^{1,p}_{\mathrm{loc}}(\Omega)$ is a solution to the equation \eqref{diveq} if and only if, for any $\varphi\in W^{1,p}_0(\Omega)$,
	
	\begin{equation}\label{diveqtest}
	\int_{\Omega}\left<A\left(x, Du(x)\right), D\varphi(x)\right>dx=-\int_{\Omega}\div A\left(x, D\psi(x)\right)\chi_{\set{u=\psi}}(x)\varphi(x)dx
	\end{equation}
	
	Let us fix a ball $B_{2R}\Subset \Omega$ and arbitrary radii $\frac{
		R}{2}<r<s<t<\lambda r<R$, with $1<\lambda<2$. Let us consider a cut
	off function $\eta\in C^\infty_0(B_t)$ such that $\eta\equiv 1$ on
	$B_s$,  $|D \eta|\le \frac{c}{R}$ and $|D^2 \eta|\le \frac{c}{R^2}$. From now on, with no
	loss of generality, we suppose $R<1$.\\
	Let us consider the test function
	$$
	\varphi(x)=\tau_{-h}\left(\eta^2(x)\tau_hu(x)\right).
	$$
	
For this choice of $\varphi$, using proposition \ref{findiffpr}, the left-hand side of \eqref{diveqtest} can be written as follows:
	
	\begin{align}\label{diveqtestL}
	&\int_{\Omega}\left<A\left(x, Du(x)\right), D\left(\tau_{-h}\left(\eta^2(x)\tau_hu(x)\right)\right)\right>dx\notag\\
	=&\int_{\Omega}\left<\tau_{h}A\left(x, Du(x)\right), D\left(\eta^2(x)\tau_hu(x)\right)\right>dx\notag\\
	=&\int_{\Omega}\left<A\left(x+h, Du(x+h)\right)-A\left(x, Du(x)\right), D\left(\eta^2(x)\tau_hu(x)\right)\right>dx\notag\\
	=&\int_{\Omega}\left<A\left(x+h, Du(x+h)\right)-A\left(x, Du(x)\right), \eta^2(x)\tau_hDu(x)\right>dx\notag\\
	&+\int_{\Omega}\left<A\left(x+h, Du(x+h)\right)-A\left(x, Du(x)\right), 2\eta(x)D\eta(x)\tau_hu(x)\right>dx\notag\\
	=&\int_{\Omega}\left<A\left(x, Du(x+h)\right)-A\left(x, Du(x)\right), \eta^2(x)\tau_hDu(x)\right>dx\notag\\
	&+\int_{\Omega}\left<A\left(x+h, Du(x+h)\right)-A\left(x, Du(x+h)\right), \eta^2(x)\tau_hDu(x)\right>dx\notag\\
	&+\int_{\Omega}\left<A\left(x+h, Du(x+h)\right)-A\left(x, Du(x)\right), 2\eta(x)D\eta(x)\tau_hu(x)\right>dx\notag\\
	:=&I_0+I+II.
\end{align}

Since the right-hand side of \eqref{diveqtest} is not zero only where $u=\psi$, using the test function given above, it becomes

\begin{equation}\label{diveqtestR*}
-\int_{\Omega}\div A(x, D\psi(x))\chi_{\set{u=\psi}}(x)\tau_{-h}\left(\eta^2(x)\tau_{h}\psi(x)\right)dx,
\end{equation}

and since the map $x\mapsto A(x, \xi)$ belongs to $W^{1,n}_{\mathrm{loc}}(\Omega)$ for any $\xi\in\R^n$, the map $\xi\mapsto A(x, \xi)$ belongs to $C^1(\R^n)$ for a. e. $x\in\Omega$ and $V_p(D\psi)\in W^{1,2}_{\mathrm{loc}}(\Omega)$, we can write \eqref{diveqtestR*} as follows

	\begin{align}\label{diveqtestR}
	&-\int_{\Omega}\Big\lbrace \Big[A_x\left(x, D\psi(x)\right)+A_\xi\left(x, D\psi(x)\right)D^2\psi(x)\Big]\chi_{\set{u=\psi}}(x)\notag\\
	&\quad\cdot\tau_{-h}\left(\eta^2(x)\tau_h\psi(x)\right)\Big\rbrace dx\notag\\
	=&-\int_{\Omega}\Big\lbrace \Big[A_x\left(x, D\psi(x)\right)+A_\xi\left(x, D\psi(x)\right)D^2\psi(x)\Big]\chi_{\set{u=\psi}}(x)\notag\\
	&\quad\cdot\tau_{-h}\left(\eta^2(x)\cdot h\int_{0}^{1}D\psi(x+h\sigma)d\sigma\right)\Big\rbrace dx\notag\\
	=&-\int_{\Omega}\Big\lbrace \Big[A_x\left(x, D\psi(x)\right)+A_\xi\left(x, D\psi(x)\right)D^2\psi(x)\Big]\chi_{\set{u=\psi}}(x)\notag\\
	&\quad\cdot h^2\int_{0}^{1}\left[\eta^2(x-h\theta)\int_{0}^{1}D^2\psi(x+h\sigma-h\theta)d\sigma\right.\notag\\
	&\left.\quad\quad+2\eta(x-h\theta)D\eta(x-h\theta)\int_{0}^{1}D\psi(x+h\sigma-h\theta)d\sigma\right]d\theta\Big\rbrace dx\notag\\
	=&-\int_{\Omega}\Big\lbrace \Big[A_x\left(x, D\psi(x)\right)+A_\xi\left(x, D\psi(x)\right)D^2\psi(x)\Big]\chi_{\set{u=\psi}}(x)\notag\\
	&\quad\cdot\int_{0}^{1}\int_{0}^{1}h^2\Big[\eta^2(x-h\theta)D^2\psi(x+h\sigma-h\theta)\notag\\
	&\quad\quad+2\eta(x-h\theta)D\eta(x-h\theta)D\psi(x+h\sigma-h\theta)\Big]d\sigma d\theta\Big\rbrace dx.
	\end{align}
	
	Therefore, the right-hand side of \eqref{diveqtest} is given by the following expression
	
	\begin{align}\label{rhs}
	&-h^2\int_{\Omega}A_x\left(x, D\psi(x)\right)\chi_{\set{u=\psi}}(x)\int_{0}^{1}\int_{0}^{1}\eta^2(x-h\theta)D^2\psi(x+h\sigma-h\theta)d\sigma d\theta dx\notag\\
	&-2h^2\int_{\Omega}A_x\left(x, D\psi(x)\right)\chi_{\set{u=\psi}}(x)\int_{0}^{1}\int_{0}^{1}\eta(x-h\theta)D\eta(x-h\theta)D\psi(x+h\sigma-h\theta)d\sigma d\theta dx\notag\\
	&-h^2\int_{\Omega}A_\xi\left(x, D\psi(x)\right)D^2\psi(x)\chi_{\set{u=\psi}}(x)\int_{0}^{1}\int_{0}^{1}\eta^2(x-h\theta)D^2\psi(x+h\sigma-h\theta)d\sigma d\theta dx\notag\\
	&-2h^2\int_{\Omega}A_\xi\left(x, D\psi(x)\right)D^2\psi(x)\chi_{\set{u=\psi}}(x)\notag\\
	&\quad\cdot\int_{0}^{1}\int_{0}^{1}\eta(x-h\theta)D\eta(x-h\theta)D\psi(x+h\sigma-h\theta)d\sigma d\theta dx\notag\\
	=:&-III-IV-V-VI.
	\end{align}
	
Inserting \eqref{diveqtestL} and \eqref{rhs} in \eqref{diveqtest} we get

\begin{align}\label{starteq}
I_0=-I-II-III-IV-V-VI,
\end{align}

and so

\begin{equation}\label{start}
I_0\le|I|+|II|+|III|+|IV|+|V|+|VI|.
\end{equation}

By assumption \eqref{p-ellipticityA1}, we have

\begin{equation}\label{I_0}
I_0\ge\nu\int_{\Omega}\eta^2(x)\left(\mu^2+\left|Du(x)\right|^2+\left|Du(x+h)\right|^2\right)^{\frac{p-2}{2}}\left|\tau_{h}Du(x)\right|^2dx.
\end{equation}

Let us consider the term $|I|$. By assumption \eqref{x-dependence}, and using Young's Inequality with exponents $\left(2, 2\right)$, H\"{o}lder's Inequality with exponents $\left(\frac{n}{2}, \frac{n}{n-2}\right)$, and the properties of $\eta$, we get

\begin{align}\label{I}
|I|\le&\int_{\Omega}|h|g(x)\left(\mu^2+\left|Du(x)\right|^2+\left|Du(x+h)\right|^2\right)^{\frac{p-1}{2}}\eta^2(x)\left|\tau_{h}Du(x)\right|dx\notag\\
\le&\varepsilon\int_{\Omega}\eta^2(x)\left(\mu^2+\left|Du(x)\right|^2+\left|Du(x+h)\right|^2\right)^{\frac{p-2}{2}}\left|\tau_{h}Du(x)\right|^2dx\notag\\
&+c_\varepsilon|h|^2\int_{\Omega}\eta^2(x)\left(\mu^2+\left|Du(x)\right|^2+\left|Du(x+h)\right|^2\right)^{\frac{p}{2}}g^2(x)dx\notag\\
\le&\varepsilon\int_{\Omega}\eta^2(x)\left(\mu^2+\left|Du(x)\right|^2+\left|Du(x+h)\right|^2\right)^{\frac{p-2}{2}}\left|\tau_{h}Du(x)\right|^2dx\notag\\
&+c_\varepsilon|h|^2\left(\int_{B_t}\left(\mu^2+\left|Du(x)\right|^2+\left|Du(x+h)\right|^2\right)^{\frac{np}{2(n-2)}}dx\right)^\frac{n-2}{n}\notag\\
&\quad\cdot\left(\int_{B_t}g^n(x)dx\right)^\frac{2}{n}.
\end{align}

For the term $II$, let us observe that, integrating by parts, for any $s=1, ..., n$, we have

\begin{align}\label{-II}
-II=&-2h\int_{\Omega}\left<\int_{0}^{1}\frac{d}{dx_s}A\left(x+h\theta e_s, Du(x+h\theta e_s)\right)d\theta, \eta(x)D\eta(x)\tau_{h}u(x)\right>dx\notag\\
=&2h\int_{\Omega}\left<\int_{0}^{1}\left(A\left(x+h\theta e_s, Du(x+h\theta e_s)\right)\right)d\theta, \frac{d}{dx_s}\left(\eta(x)D\eta(x)\tau_{h}u(x)\right)\right>dx
\end{align}

so we can estimate $II$ as follows

\begin{align}\label{II*}
|II|\le&2|h|\int_{\Omega}\int_{0}^{1}\left|A\left(x+h\theta, Du(x+h\theta)\right)\right|\Big(\left|D\eta(x)\right|^2\left|\tau_{h}u(x)\right|\notag\\
&\quad+\eta(x)\left|D^2\eta(x)\right|\left|\tau_{h}u(x)\right|\Big)d\theta dx\notag\\
&+2|h|\int_{\Omega}\int_{0}^{1}\left|A\left(x+h\theta, Du(x+h\theta)\right)\right|\Big(\eta(x)\left|D\eta(x)\right|\left|\tau_{h}Du(x)\right|\Big)d\theta dx\notag\\
\le&2|h|\int_{\Omega}\int_{0}^{1}\left|A\left(x+h\theta, Du(x+h\theta)\right)\right|\Big(\left|D\eta(x)\right|^2\notag\\
&\quad+\eta(x)\left|D^2\eta(x)\right|\Big)d\theta \left|\tau_{h}u(x)\right|dx\notag\\
&+2|h|\int_{\Omega}\int_{0}^{1}\left|A\left(x+h\theta, Du(x+h\theta)\right)\right|\eta(x)\left|D\eta(x)\right|\left|\tau_{h}Du(x)\right|d\theta dx.
\end{align}

Now, recalling the properties of $\eta$, assumption \eqref{p-growthA3}, and using H\"{o}lder's Inequality with exponents $\left(p, \frac{p}{p-1}\right)$ and Young's Inequality with exponents $\left(2, 2\right)$, we get

\begin{align}\label{II**}
|II|\le&2|h|\int_{\Omega}\int_{0}^{1}\left(\mu^2+\left|Du(x)\right|^2+\left|Du(x+h\theta)\right|^2\right)^\frac{p-1}{2}\Big(\left|D\eta(x)\right|^2\notag\\
&\quad+\eta(x)\left|D^2\eta(x)\right|\Big)d\theta\left|\tau_{h}u(x)\right|dx\notag\\
&+2|h|\int_{\Omega}\int_{0}^{1}\left(\mu^2+\left|Du(x)\right|^2+\left|Du(x+h\theta)\right|^2\right)^\frac{p-1}{2}\eta(x)\left|D\eta(x)\right|\left|\tau_{h}Du(x)\right|d\theta dx\notag\\
=&2|h|\int_{0}^{1}\int_{\Omega}\left(\mu^2+\left|Du(x)\right|^2+\left|Du(x+h\theta)\right|^2\right)^\frac{p-1}{2}\Big(\left|D\eta(x)\right|^2\notag\\
&\quad+\eta(x)\left|D^2\eta(x)\right|\Big)\left|\tau_{h}u(x)\right|dxd\theta\notag\\
&+2|h|\int_{0}^{1}\int_{\Omega}\left(\mu^2+\left|Du(x)\right|^2+\left|Du(x+h\theta)\right|^2\right)^\frac{p-1}{2}\eta(x)\left|D\eta(x)\right|\left|\tau_{h}Du(x)\right|dxd\theta\notag\\
\le&\frac{c|h|}{R^2}\int_0^1\left(\int_{B_t}\left(\mu^2+\left|Du(x)\right|^2+\left|Du(x+h\theta)\right|^2\right)^\frac{p}{2}dx\right)^\frac{p-1}{p}d\theta\notag\\
&\quad\cdot\left(\int_{B_t}\left|\tau_{h}u(x)\right|^pdx\right)^\frac{1}{p}\notag\\
&+\varepsilon\int_{\Omega}\eta^2(x)\left|\tau_{h}Du(x)\right|^2\left(\mu^2+\left|Du(x)\right|^2+\left|Du(x+h)\right|^2\right)^\frac{p-2}{2}dx\notag\\
&+\frac{c_\varepsilon\left|h\right|^2}{R^2}\int_{0}^{1}\int_{B_t}\left(\mu^2+\left|Du(x)\right|^2+\left|Du(x+\theta h)\right|^2\right)^{p-1}\notag\\
&\quad\cdot\left(\mu^2+\left|Du(x)\right|^2+\left|Du(x+h)\right|^2\right)^\frac{2-p}{2}dxd\theta.
\end{align}

Now, by Lemma \ref{Giusti8.2}, we get

\begin{align}\label{II}
|II|\le&\frac{c|h|^2}{R^2}\int_0^1\left(\int_{B_t}\left(\mu^2+\left|Du(x)\right|^2+\left|Du(x+h\theta)\right|^2\right)^\frac{p}{2}dx\right)^\frac{p-1}{p}d\theta\notag\\
&\quad\cdot\left(\int_{B_t}\left|Du(x)\right|^pdx\right)^\frac{1}{p}\notag\\
&+\varepsilon\int_{\Omega}\eta^2(x)\left|\tau_{h}Du(x)\right|^2\left(\mu^2+\left|Du(x)\right|^2+\left|Du(x+h)\right|^2\right)^\frac{p-2}{2}dx\notag\\
&+\frac{c_\varepsilon\left|h\right|^2}{R^2}\int_{0}^{1}\left[\int_{\Omega}\left(\mu^2+\left|Du(x)\right|^2+\left|Du(x+\theta h)\right|^2\right)^\frac{p-1}{2}\right.\notag\\
&\left.\quad\cdot\left(\mu^2+\left|Du(x)\right|^2+\left|Du(x+h)\right|^2\right)^\frac{2-p}{4}dx\right]^2d\theta.
\end{align}

Let us consider, now, the term $III$.
By \eqref{x-dependence} and the properties of $\eta$, we get

\begin{align}\label{III*}
|III|\le&|h|^2\int_{0}^{1}\int_{0}^{1}\int_{B_{\lambda r}}g(x)\left(\mu^2+\left|D\psi(x)\right|^2+\left|D\psi(x+h)\right|^2\right)^\frac{p-1}{2}\notag\\
&\quad\cdot\left|D^2\psi(x+h\sigma-h\theta)\right|dxd\sigma d\theta.
\end{align}

Using Young's Inequality with exponents $\left(2, 2\right)$, we get

\begin{align}\label{III**}
|III|\le&c|h|^2\int_{0}^{1}\int_{0}^{1}\left[\int_{B_{\lambda r}} g^2(x)\left(\mu^2+\left|D\psi(x)\right|^2+\left|D\psi(x+h)\right|^2\right)^\frac{p}{2}dx\right.\notag\\
&\left.\quad+\int_{B_{\lambda r}}\left(\mu^2+\left|D\psi(x)\right|^2+\left|D\psi(x+h)\right|^2\right)^\frac{p-2}{2}\right.\notag\\
&\left.\quad\quad\cdot\left|D^2\psi(x+h\sigma-h\theta)\right|^2dx\right]d\sigma d\theta.
\end{align}

Using Young's Inequality with exponents $\left(\frac{n}{2}, \frac{n}{n-2}\right)$ in the first integral of \eqref{III**}, we get 

\begin{align}\label{III}
|III|\le&c|h|^2\left[\int_{B_{\lambda r}}g^n(x)dx+\int_{B_{\lambda r}}\left(\mu^2+\left|D\psi(x)\right|^2+\left|D\psi(x+h)\right|^2\right)^\frac{np}{2(n-2)}dx\right]\notag\\
&\quad+c\int_0^1\int_0^1\int_{B_{\lambda r}}\left(\mu^2+\left|D\psi(x)\right|^2+\left|D\psi(x+h)\right|^2\right)^\frac{p-2}{2}\left|D^2\psi(x+h\sigma-h\theta)\right|^2dxd\sigma d\theta.
\end{align}

We estimate the term $IV$ using assumption \eqref{x-dependence}, thus getting

\begin{align}\label{IV}
|IV|\le&2|h|^2\int_{B_{\lambda r}} g(x)\left(\mu^2+\left|D\psi(x)\right|^2+\left|D\psi(x+h)\right|^2\right)^\frac{p-1}{2}\notag\\
&\quad\cdot\int_{0}^{1}\int_{0}^{1}\left|D\psi(x+h\sigma-h\theta)\right|\left|D\eta(x-h\theta)\right|d\sigma d\theta dx.
\end{align}

Let us consider, now, the term $V$. By assumption \eqref{A2bis}, we get

\begin{align}\label{V}
|V|\le&|h|^2\int_{B_{\lambda r}} \left(\mu^2+\left|D\psi(x)\right|^2+\left|D\psi(x+h)\right|^2\right)^\frac{p-2}{2}\left|D^2\psi(x)\right|\notag\\
&\quad\cdot\int_{0}^{1}\int_{0}^{1}\left|D^2\psi(x+h\sigma-h\theta)\right|d\sigma d\theta dx.
\end{align}

In order to estimate the term $VI$, we recall \eqref{p-growtA2} again, thus getting

\begin{align}\label{VI}
|VI|\le&2|h|^2\int_{B_{\lambda r}}\left(\mu^2+\left|D\psi(x)\right|^2+\left|D\psi(x+h)\right|^2\right)^\frac{p-2}{2}\left|D^2\psi(x)\right|\notag\\
&\quad\cdot\int_{0}^{1}\int_{0}^{1}\left|D\eta(x-h\theta)\right|\left|D\psi(x+h\sigma-h\theta)\right|d\sigma d\theta dx.
\end{align}

Now, plugging \eqref{I_0}, \eqref{I}, \eqref{II}, \eqref{III}, \eqref{IV}, \eqref{V} and \eqref{VI} in \eqref{start}, recalling the properties of $\eta$ and choosing a sufficiently small value of $\varepsilon$, we get

\begin{align}\label{full1}
\int_{B_R}&\eta^2(x)\left(\mu^2+\left|Du(x)\right|^2+\left|Du(x+h)\right|^2\right)^{\frac{p-2}{2}}\left|\tau_{h}Du(x)\right|^2dx\notag\\
\le&c|h|^2\left(\int_{B_t}\left(\mu^2+\left|Du(x)\right|^2+\left|Du(x+h)\right|^2\right)^{\frac{np}{2(n-2)}}dx\right)^\frac{n-2}{n}\cdot\left(\int_{B_t}g^n(x)dx\right)^\frac{2}{n}\notag\\
&+\frac{c|h|^2}{R^2}\int_0^1\left(\int_{B_t}\left(\mu^2+\left|Du(x)\right|^2+\left|Du(x+h\theta)\right|^2\right)^\frac{p}{2}dx\right)^\frac{p-1}{p}d\theta\notag\\
&\quad\cdot\left(\int_{B_t}\left|Du(x)\right|^pdx\right)^\frac{1}{p}\notag\\
&+\frac{c\left|h\right|^2}{R^2}\int_{0}^{1}\int_{B_t}\left(\mu^2+\left|Du(x)\right|^2+\left|Du(x+\theta h)\right|^2\right)^{p-1}\notag\\
&\quad\cdot\left(\mu^2+\left|Du(x)\right|^2+\left|Du(x+h)\right|^2\right)^\frac{2-p}{2}dxd\theta\notag\\
&+c|h|^2\left[\int_{B_{\lambda r}}g^n(x)dx+\int_{B_{\lambda r}}\left(\mu^2+\left|D\psi(x)\right|^2+\left|D\psi(x+h)\right|^2\right)^\frac{np}{2(n-2)}dx\right]\notag\\
&\quad+c|h|^2\int_0^1\int_0^1\int_{B_{\lambda r}}\left(\mu^2+\left|D\psi(x)\right|^2+\left|D\psi(x+h)\right|^2\right)^\frac{p-2}{2}\left|D^2\psi(x+h\sigma-h\theta)\right|^2dxd\sigma d\theta\notag\\
&+2|h|^2\int_{0}^{1}\int_{0}^{1}\int_{B_{\lambda r}} g(x)\left(\mu^2+\left|D\psi(x)\right|^2+\left|D\psi(x+h)\right|^2\right)^\frac{p-1}{2}\notag\\
&\quad\cdot\left|D\psi(x+h\sigma-h\theta)\right|\left|D\eta(x-h\theta)\right|dxd\sigma d\theta\notag\\
&+|h|^2\int_{0}^{1}\int_{0}^{1}\int_{B_{\lambda r}} \left(\mu^2+\left|D\psi(x)\right|^2+\left|D\psi(x+h)\right|^2\right)^\frac{p-2}{2}\left|D^2\psi(x)\right|\notag\\
&\quad\cdot\left|D^2\psi(x+h\sigma-h\theta)\right|dxd\sigma d\theta \notag\\
&+2|h|^2\int_{0}^{1}\int_{0}^{1}\int_{B_{\lambda r}}\left(\mu^2+\left|D\psi(x)\right|^2+\left|D\psi(x+h)\right|^2\right)^\frac{p-2}{2}\left|D^2\psi(x)\right|\notag\\
&\quad\cdot\left|D\eta(x-h\theta)\right|\left|D\psi(x+h\sigma-h\theta)\right|dxd\sigma d\theta.
\end{align}

By Lemma \ref{lemma6GP} and the properties of $\eta$, the left-hand side of \eqref{full1} can be bounded from below as follows

\begin{equation}\label{Vpbelow}
\int_{B_R}\eta^2(x)\left(\mu^2+\left|Du(x)\right|^2+\left|Du(x+h)\right|^2\right)^{\frac{p-2}{2}}\left|\tau_{h}Du(x)\right|^2dx\ge\int_{B_{\frac{R}{2}}}\left|\tau_{h}V_p(Du(x))\right|^2dx.
\end{equation}

So, by \eqref{Vpbelow} and \eqref{full1}, recalling the properties of $\eta$ and using Lemma \ref{le1}, we get 

\begin{align}\label{full2}
&\int_{B_{\frac{R}{2}}}\left|\tau_{h}V_p(Du(x))\right|^2dx\notag\\
\le& c|h|^2\left(\int_{B_R}\left(\mu^2+\left|Du(x)\right|^2\right)^{\frac{np}{2(n-2)}}dx\right)^\frac{n-2}{n}\cdot\left(\int_{B_R}g^n(x)dx\right)^\frac{2}{n}\notag\\
&+\frac{c|h|^2}{R^2}\int_{B_R}\left(\mu^2+\left|Du(x)\right|^2\right)^\frac{p}{2}dx\notag\\
&+c|h|^2\left[\int_{B_{R}}g^n(x)dx+\int_{B_{\lambda r}}\left(\mu^2+\left|D\psi(x)\right|^2\right)^\frac{np}{2(n-2)}dx\right]\notag\\
&+\frac{c|h|^2}{R}\int_{B_{R}} g(x)\left(\mu^2+\left|D\psi(x)\right|^2\right)^\frac{p}{2}dx\notag\\
&+c|h|^2\int_{B_{R}} \left(\mu^2+\left|D\psi(x)\right|^2\right)^\frac{p-2}{2}\left|D^2\psi(x)\right|^2dx\notag\\
&+\frac{c|h|^2}{R}\int_{B_{R}}\left(\mu^2+\left|D\psi(x)\right|^2\right)^\frac{p-1}{2}\left|D^2\psi(x)\right|dx.
\end{align}

Now we apply H\"{o}lder's Inequality with three exponents $\left(n, n, \frac{n}{n-2}\right)$ to the integral of the fifth line, Young's Inequality with exponents $\left(2, 2\right)$ to the last integral, and use Lemma \ref{lemma6GP}, thus getting

\begin{align}\label{full4}
&\int_{B_{\frac{R}{2}}}\left|\tau_{h}V_p(Du(x))\right|^2dx\notag\\
\le& c|h|^2\left(\int_{B_R}\left(\mu^2+\left|Du(x)\right|^2\right)^{\frac{np}{2(n-2)}}dx\right)^\frac{n-2}{n}\cdot\left(\int_{B_R}g^n(x)dx\right)^\frac{2}{n}\notag\\
&+\frac{c|h|^2}{R^2}\left(\int_{B_R}\left(\mu^2+\left|Du(x)\right|^2\right)^\frac{p}{2}dx\right)\notag\\
&+c|h|^2\left[\int_{B_{R}}g^n(x)dx+\int_{B_{R}}\left(\mu^2+\left|D\psi(x)\right|^2\right)^\frac{np}{2(n-2)}dx\right]\notag\\
&+c|h|^2\left(\int_{B_{R}} g^n(x)dx\right)^\frac{1}{n}\cdot\left(\int_{B_{R}}\left(\mu^2+\left|D\psi(x)\right|^2\right)^\frac{np}{2(n-2)}dx\right)^\frac{n-2}{n}\notag\\
&+\frac{c|h|^2}{R}\left[\int_{B_{R}}\left|DV_p\left(D\psi(x)\right)\right|^2dx+\int_{B_{R}}\left(\mu^2+\left|D\psi(x)\right|^2\right)^\frac{p}{2}dx\right]
\end{align}

for a suitable constant $c=c(n, p, \nu, L, \ell)$. By Young's Inequality with exponents $\left(\frac{n}{2}, \frac{n}{n-2}\right)$, we get

\begin{align}\label{conclusion1.2}
&\int_{B_{\frac{R}{2}}}\left|\tau_{h}V_p(Du(x))\right|^2dx\notag\\
\le& c|h|^2\left[\int_{B_R}\left(\mu^2+\left|Du(x)\right|^2\right)^{\frac{np}{2(n-2)}}dx+\int_{B_R}g^n(x)dx+\left(\int_{B_{R}} g^n(x)dx\right)^\frac{1}{2}\right.\notag\\
&\left.\quad+\int_{B_R}\left(\mu^2+\left|D\psi(x)\right|^2\right)^{\frac{np}{2(n-2)}}dx\right]\notag\\
&+\frac{c|h|^2}{R}\left[\int_{B_{R}}\left|DV_p\left(D\psi(x)\right)\right|^2dx+\int_{B_{R}}\left(\mu^2+\left|D\psi(x)\right|^2\right)^\frac{p}{2}dx\right]\notag\\
&+\frac{c|h|^2}{R^2}\left(\int_{B_R}\left(\mu^2+\left|Du(x)\right|^2\right)^\frac{p}{2}dx\right).
\end{align}

Let us observe that, since $V_p\left(D\psi\right)\in W^{1,2}_{\mathrm{loc}}\left(\Omega\right)$, by Sobolev's Inequality, $D\psi\in L^\frac{np}{n-2}_{\mathrm{loc}}\left(\Omega\right)$.
Therefore, applying Theorem \ref{ThmByunChoOk} with $q=\frac{np}{n-2}$, we have $Du\in L^\frac{np}{n-2}_{\mathrm{loc}}\left(\Omega\right),$ with the following estimate:

\begin{align}\label{BCOestimate}
&\Mint_{B_{R}}\left|Du(x)\right|^\frac{np}{n-2}dx\le C\left\lbrace 1+\Mint_{B_{2R}}\left|D\psi(x)\right|^\frac{np}{n-2}dx+\left(\Mint_{B_{2R}}\left|Du(x)\right|^pdx\right)^\frac{n}{n-2}\right\rbrace.
\end{align}

Using \eqref{BCOestimate}, estimate \eqref{conclusion1.2} becomes

\begin{align}\label{conclusion1.3}
&\int_{B_{\frac{R}{2}}}\left|\tau_{h}V_p(Du(x))\right|^2dx\notag\\
\le&c|h|^2\left\lbrace\int_{B_{2R}}\left(\mu^2+\left|D\psi(x)\right|^2\right)^{\frac{np}{2(n-2)}}dx+\left[\int_{B_{2R}}\left(\mu^2+\left|Du(x)\right|^2\right)^{\frac{p}{2}}dx\right]^{\frac{n}{n-2}}\right.\notag\\
&\left.\quad+\int_{B_R}g^n(x)dx+\left(\int_{B_{R}} g^n(x)dx\right)^\frac{1}{2}dx\right\rbrace\notag\\
&+\frac{c|h|^2}{R}\left[\int_{B_{R}}\left|DV_p\left(D\psi(x)\right)\right|^2dx+\int_{B_{R}}\left(\mu^2+\left|D\psi(x)\right|^2\right)^\frac{p}{2}dx\right]\notag\\
&+\frac{c|h|^2}{R^2}\left(\int_{B_R}\left(\mu^2+\left|Du(x)\right|^2\right)^\frac{p}{2}dx\right).
\end{align}

Applying Sobolev's embedding Theorem to the function $V_p\left(D\psi\right)$, and exploiting the fact that $p<\frac{np}{n-p}$, we get

\begin{align}\label{conclusion1.3*}
&\int_{B_{\frac{R}{2}}}\left|\tau_{h}V_p(Du(x))\right|^2dx\notag\\
\le&c|h|^2\left\lbrace\left[\int_{B_{2R}}\left(\left|V_p\left(D\psi(x)\right)\right|^2+\left|DV_p\left(D\psi(x)\right)\right|^2\right)dx\right]^\frac{n}{n-2}+\left[\int_{B_{2R}}\left(\mu^2+\left|Du(x)\right|^2\right)^{\frac{p}{2}}dx\right]^{\frac{n}{n-2}}\right.\notag\\
&\left.\quad+\int_{B_R}g^n(x)dx+\left(\int_{B_{R}} g^n(x)dx\right)^\frac{1}{2}dx\right\rbrace\notag\\
&+\frac{c|h|^2}{R}\int_{B_{R}}\left|DV_p\left(D\psi(x)\right)\right|^2dx+\frac{c|h|^2}{R^2}\left(\int_{B_R}\left(\mu^2+\left|Du(x)\right|^2\right)^\frac{p}{2}dx\right).
\end{align}

So, applying Lemma \ref{Giusti8.2}, for a suitable choice of $C$ and $\sigma$, we get

\begin{equation}\label{estimate1pf}
\left\Arrowvert DV_p(Du(x))\right\Arrowvert_{L^2\left(B_{\frac{R}{2}}\right)}\le C\left(1+\left\Arrowvert Du\right\Arrowvert_{L^p(B_{2R})}+\left\Arrowvert V_p\left(D\psi\right)\right\Arrowvert_{W^{1, 2}(B_{2R})}+\left\Arrowvert g\right\Arrowvert_{L^n\left(B_{R}\right)}\right)^\sigma,
\end{equation}

that is the conclusion.
\end{proof}

\section{Proof of Theorem \ref{thm2}}\label{Thm2Pf}

This section is devoted to the proof of Theorem \ref{thm2}. It is worth noticing that, in this case, our starting point can't be equation \eqref{diveq}, since our assumption on $\psi$ doesn't allow to calculate the divergence in the right-hand side.

\begin{proof}
Let us fix a ball $B_{4R}\Subset \Omega$ and arbitrary radii $\frac{R}{2}<r<s<t<\lambda r<R$, with $1<\lambda<2$. Let us consider a cut-off function $\eta\in C^\infty_0(B_t)$ such that $\eta\equiv 1$ on $B_s$ and  $|D \eta|\le \frac{c}{t-s}$. From now on, with no loss of generality, we suppose $R<1$.\\
Let $v\in W^{1, p}_0(\Omega)$ be such that

\begin{equation}\label{cond}
u-\psi+\tau v\ge 0\qquad\forall \tau\in[0, 1],
\end{equation}

\noindent  and observe that $\varphi:=u+\tau v\in
\mathcal{K}_\psi(\Omega)$ for all $\tau \in [0, 1]$, since
$\varphi=u+\tau v\ge \psi $. For $|h|<\frac{R}{4}$, we consider

\begin{equation}\label{v1}
v_1(x)=\eta^2(x)\left[(u-\psi)(x+h)-(u-\psi)(x)\right],
\end{equation}

\noindent  so we have  $v_1\in W^{1, p}_0(\Omega)$, and,
for any $\tau\in[0,1]$, $v_1$ satisfies \eqref{cond}. Indeed, for
a. e. $x \in \Omega$ and for any $\tau\in[0,1]$

\begin{align*}
	u(x)-\psi(x)+\tau v_1(x)=&
	u(x)-\psi(x)+\tau\eta^2(x)\left[(u-\psi)(x+h)-(u-\psi)(x)\right] \notag\\
	=&\tau\eta^2(x)(u-\psi)(x+h)+(1-\tau\eta^2(x))(u-\psi)(x)\ge 0,
\end{align*}

\noindent   since $u\in \mathcal{K}_\psi(\Omega)$ and $0\le\eta\le1$.\\
So we can use $\varphi=u+\tau v_1$ as a test function in inequality \eqref{variationalinequality},  thus getting

\begin{equation}\label{3.3}
0\le\int_{\Omega}\left<A(x, Du(x)), D\left[\eta^2(x)\left[(u-\psi)(x+h)-(u-\psi)(x)\right]\right]\right>dx.
\end{equation}

In a similar way, we define

\begin{equation}\label{v2}
v_2(x)=\eta^2(x-h)\left[(u-\psi)(x-h)-(u-\psi)(x)\right],
\end{equation}

and we have $v_2\in W^{1, p}_0(\Omega)$, and
\eqref{cond} still is satisfied for any $\tau\in[0,1]$, since

\begin{align*}
	u(x)-\psi(x)+\tau v_2(x)=&
	u(x)-\psi(x)+\tau\eta^2(x-h)\left[(u-\psi)(x-h)-(u-\psi)(x)\right]
	\notag\\=&\tau\eta^2(x)(u-\psi)(x-h)+(1-\tau\eta^2(x-h))(u-\psi)(x)\ge
	0.
\end{align*}

By using in \eqref{variationalinequality} as test function
$\varphi=u+\tau v_2$, we get

\begin{equation}
0\le\int_{\Omega}\left<A(x, Du(x)), D\left[\eta^2(x-h)\left[(u-\psi)(x-h)-(u-\psi)(x)\right]\right]\right>dx,
\end{equation}

and by means of a change of variable, we obtain

\begin{equation}\label{3.5}
0\le\int_{\Omega}\left<A(x+h, Du(x+h)), D\left[\eta^2(x)\left[(u-\psi)(x)-(u-\psi)(x+h)\right]\right]\right>dx.
\end{equation}

We can add \eqref{3.3} and \eqref{3.5}, thus getting

\begin{align*}
0\le&\int_{\Omega}\left<A(x, Du(x)), D\left[\eta^2(x)\left[(u-\psi)(x+h)-(u-\psi)(x)\right]\right]\right>dx\\
&+\int_{\Omega}\left<A(x+h, Du(x+h)), D\left[\eta^2(x)\left[(u-\psi)(x)-(u-\psi)(x+h)\right]\right]\right>dx,
\end{align*}

that is

\begin{equation*}
0\le\int_{\Omega}\left<A(x, Du(x))-A(x+h, Du(x+h)), D\left[\eta^2(x)\left[(u-\psi)(x+h)-(u-\psi)(x)\right]\right]\right>dx,
\end{equation*}


which implies

\begin{align*}
0\ge&\int_{\Omega}\left<A(x+h, Du(x+h))-A(x, Du(x)), \eta^2(x)D\left[(u-\psi)(x+h)-(u-\psi)(x)\right]\right>dx\notag\\
&+\int_{\Omega}\left<A(x+h, Du(x+h))-A(x, Du(x)),
2\eta(x)D\eta(x)\left[(u-\psi)(x+h)-(u-\psi)(x)\right]\right>dx.
\end{align*}

Previous inequality can be rewritten as follows

\begin{align}\label{differenceinequality3}
0\ge&\int_{\Omega}\left<A(x+h, Du(x+h))-A(x+h, Du(x)),\eta^2(x)(Du(x+h)-Du(x))\right>dx\notag\\
&-\int_{\Omega}\left<A(x+h, Du(x+h))-A(x+h, Du(x)), \eta^2(x)(D\psi(x+h)-D\psi(x))\right>dx\notag\\
&+\int_{\Omega}\left<A(x+h, Du(x+h))-A(x+h, Du(x)), 2\eta(x)D\eta(x)\tau_h\left(u-\psi\right)(x)\right>dx\notag\\
&+\int_{\Omega}\left<A(x+h, Du(x))-A(x, Du(x)),\eta^2(x)(Du(x+h)-Du(x))\right>dx\notag\\
&-\int_{\Omega}\left<A(x+h, Du(x))-A(x, Du(x)), \eta^2(x)(D\psi(x+h)-D\psi(x))\right>dx\notag\\
&+\int_{\Omega}\left<A(x+h, Du(x))-A(x, Du(x)), 2\eta(x)D\eta(x)\tau_h\left(u-\psi\right)(x)\right>dx\notag\\
=:& \,I+II+III+IV+V+VI,
\end{align}

\noindent so we have

\begin{equation}\label{differenceinequality}
I\le |II|+|III|+|IV|+|V|+|VI|.
\end{equation}

By \eqref{p-ellipticityA1} we have

\begin{equation}\label{I_b}
I\ge\nu\int_{\Omega}\eta^2(x)\left(\mu^2+\left|Du(x)\right|^2+\left|Du(x+h)\right|^2\right)^\frac{p-2}{2}\left|\tau_hDu(x)\right|^2dx.
\end{equation}

Before going further, let us observe that, since $V_p\left(D\psi\right)\in B^\alpha_{2,q,\mathrm{loc}}(\Omega)$ with $q\le 2^*_\alpha$ then, by Lemma \ref{EP2.2}, $V_p\left(D\psi\right)\in L^\frac{2n}{n-2\alpha}_\mathrm{loc}(\Omega)$, and so $D\psi\in L^\frac{np}{n-2\alpha}_\mathrm{loc}(\Omega)$ and, by Theorem \ref{ThmByunChoOk}, we also have $Du\in L^\frac{np}{n-2\alpha}_\mathrm{loc}(\Omega)$.\\
Let us consider the term $II$. By assumption \eqref{p-growtA2} we have

\begin{equation}\label{II_b'}
\left|II\right|\le L\int_{\Omega}\eta^2(x)\left(\mu^2+\left|Du(x)\right|^2+\left|Du(x+h)\right|^2\right)^\frac{p-2}{2}\left|\tau_hDu(x)\right|\left|\tau_hD\psi(x)\right|.
\end{equation}

Now we set 

\[
E_1:=\LBrace x\in\Omega: \lAbs Du(x)\rAbs^2+\lAbs Du(x+h)\rAbs^2>\lAbs D\psi(x)\rAbs^2+\lAbs D\psi(x+h)\rAbs^2\RBrace
\]

and

\[
E_2:=\Omega\setminus E_1, 
\]

so \eqref{II_b'} becomes

\begin{align}\label{II_b''}
\left|II\right|\le& L\int_{E_1}\eta^2(x)\left(\mu^2+\left|Du(x)\right|^2+\left|Du(x+h)\right|^2\right)^\frac{p-2}{2}\left|\tau_hDu(x)\right|\left|\tau_hD\psi(x)\right|\notag\\
&+L\int_{E_2}\eta^2(x)\left(\mu^2+\left|Du(x)\right|^2+\left|Du(x+h)\right|^2\right)^\frac{p-2}{2}\left|\tau_hDu(x)\right|\left|\tau_hD\psi(x)\right|\notag\\
=:&II_1+II_2.
\end{align}

Since $1<p<2$, using Young's Inequality with exponents $\left(2, 2\right)$, the properties of $\eta$ and Lemma \ref{lemma6GP}, we get

\begin{align}\label{II_1}
II_1\le& L\int_{E_1}\eta^2(x)\left(\mu^2+\left|Du(x)\right|^2+\left|Du(x+h)\right|^2\right)^\frac{p-2}{4}\left|\tau_hDu(x)\right|\notag\\
&\cdot \left(\mu^2+\left|D\psi(x)\right|^2+\left|D\psi(x+h)\right|^2\right)^\frac{p-2}{4}\left|\tau_hD\psi(x)\right|dx\notag\\
\le&\varepsilon\int_{E_1}\eta^2(x)\left(\mu^2+\left|Du(x)\right|^2+\left|Du(x+h)\right|^2\right)^\frac{p-2}{2}\left|\tau_hDu(x)\right|^2dx\notag\\
&+c_\varepsilon\int_{E_1}\eta^2(x)\left(\mu^2+\left|D\psi(x)\right|^2+\left|D\psi(x+h)\right|^2\right)^\frac{p-2}{2}\left|\tau_hD\psi(x)\right|^2dx\notag\\
\le&\varepsilon\int_{\Omega}\eta^2(x)\left(\mu^2+\left|Du(x)\right|^2+\left|Du(x+h)\right|^2\right)^\frac{p-2}{2}\left|\tau_hDu(x)\right|^2dx\notag\\
&+c_\varepsilon\int_{B_t}\lAbs\tau_hV_p\left(D\psi(x)\right)\rAbs^2dx.
\end{align}

For what concerns the term $II_2$, using Young's Inequality with exponents $\left(2, 2\right)$, the properties of $\eta$, and Lemmas \ref{lemma6GP} and \ref{le1}, we have

\begin{align}\label{II_2}
II_2\le& L\int_{E_2}\eta^2(x)\left(\mu^2+\left|Du(x)\right|^2+\left|Du(x+h)\right|^2\right)^\frac{p-1}{2}\left|\tau_hD\psi(x)\right|\notag\\
\le&L\int_{E_2}\eta^2(x)\left(\mu^2+\left|D\psi(x)\right|^2+\left|D\psi(x+h)\right|^2\right)^\frac{p-1}{2}\left|\tau_hD\psi(x)\right|\notag\\
\le&c\int_{B_t}\left(\mu^2+\left|D\psi(x)\right|^2+\left|D\psi(x+h)\right|^2\right)^\frac{p-2}{2}\left|\tau_hD\psi(x)\right|^2dx\notag\\
&+c\int_{B_t}\left(\mu^2+\left|D\psi(x)\right|^2+\left|D\psi(x+h)\right|^2\right)^\frac{p}{2}dx\notag\\
\le&c\int_{B_t}\lAbs\tau_hV_p\left(D\psi(x)\right)\rAbs^2dx+c\int_{B_{\lambda r}}\left(\mu^p+\left|D\psi(x)\right|^p\right)dx\notag\\
\le& c\int_{B_t}\lAbs\tau_hV_p\left(D\psi(x)\right)\rAbs^2dx+cR^{2\alpha}\left[\int_{B_{R}}\left(1+\left|D\psi(x)\right|^\frac{np}{n-2\alpha}\right)dx\right]^\frac{n-2\alpha}{n},
\end{align}

where we used the fact that $D\psi\in\L^\frac{np}{n-2\alpha}$, and since $p<\frac{np}{n-2\alpha}$ we have

\[
\int_{B_{R}}\left|D\psi(x)\right|^pdx\le\left(\omega_nR^n\right)^{1-\frac{n-2\alpha}{n}}\left(\int_{B_{R}}\left|D\psi(x)\right|^\frac{np}{n-2\alpha}dx\right)^\frac{n-2\alpha}{n},
\]

where $\omega_n$ is the measure of the ball of radius 1 in $\R^n$.\\
Plugging \eqref{II_1} and \eqref{II_2} into \eqref{II_b''}, we get
the following estimate for the term $II$:

\begin{align}\label{II_b}
\lAbs II\rAbs\le&\varepsilon\int_{\Omega}\eta^2(x)\left(\mu^2+\left|Du(x)\right|^2+\left|Du(x+h)\right|^2\right)^\frac{p-2}{2}\left|\tau_hDu(x)\right|^2dx\notag\\
&+ c_\varepsilon\int_{B_t}\lAbs\tau_hV_p\left(D\psi(x)\right)\rAbs^2dx+cR^{2\alpha}\left[\int_{B_{R}}\left(1+\left|D\psi(x)\right|^\frac{np}{n-2\alpha}\right)dx\right]^\frac{n-2\alpha}{n}.
\end{align}

Now we consider the term $III$. By assumption \eqref{p-growtA2}, Young's Inequality with exponents $\left(p, \frac{p}{p-1}\right)$ the fact that $1<p<2$ and the properties of $\eta$ we have

\begin{align}\label{III_b'}
\left|III\right|\le&L\int_{\Omega}\eta(x)\lAbs D\eta(x)\rAbs\left(\mu^2+\left|Du(x)\right|^2+\left|Du(x+h)\right|^2\right)^{\frac{p-2}{2}}\lAbs\tau_{h}Du(x)\rAbs\lAbs\tau_{h}(u-\psi)(x)\rAbs dx\notag\\
\le&\varepsilon\int_{\Omega}\eta^\frac{p}{p-1}(x)\left(\mu^2+\left|Du(x)\right|^2+\left|Du(x+h)\right|^2\right)^{\frac{p-2}{2}\cdot\frac{p}{p-1}}\left|\tau_hDu(x)\right|^{\frac{p}{p-1}-2}\cdot\left|\tau_hDu(x)\right|^{2}dx\notag\\
&+\frac{c_\varepsilon}{R^p}\int_{B_{R}}\left|\tau_{h}(u-\psi)(x)\right|^pdx\notag\\
\le&\varepsilon\int_{\Omega}\eta^\frac{p}{p-1}(x)\left(\mu^2+\left|Du(x)\right|^2+\left|Du(x+h)\right|^2\right)^{\frac{p-2}{2}}\left|\tau_hDu(x)\right|^{2}dx\notag\\
&+\frac{c_\varepsilon}{R^p}\int_{B_{R}}\left|\tau_{h}(u-\psi)(x)\right|^pdx,
\end{align}

and using Lemma \ref{le1} we get

\begin{align}\label{III_b}
\left|III\right|\le&\varepsilon\int_{\Omega}\eta^\frac{p}{p-1}(x)\left(\mu^2+\left|Du(x)\right|^2+\left|Du(x+h)\right|^2\right)^{\frac{p-2}{2}}\left|\tau_hDu(x)\right|^2dx\notag\\
&+\frac{c_\varepsilon|h|^p}{R^p}\int_{B_{\lambda R}}\left|D(u-\psi)(x)\right|^pdx\notag\\
\le&\varepsilon\int_{\Omega}\eta^\frac{p}{p-1}(x)\left(\mu^2+\left|Du(x)\right|^2+\left|Du(x+h)\right|^2\right)^{\frac{p-2}{2}}\left|\tau_hDu(x)\right|^2dx\notag\\
&+\frac{c_\varepsilon|h|^p}{R^{p-2\alpha}}\left(\int_{B_{2R}}\left|D(u-\psi)(x)\right|^\frac{np}{n-2\alpha}dx\right)^\frac{n-2\alpha}{n},
\end{align}
where, in the last line, we used the fact that $D(u-\psi)\in L^{\frac{np}{n-2\alpha}}$, arguing like in \eqref{II_2}.\\
Let us consider, now, the term $IV$. By \eqref{A5}, Young's Inequality with exponents $\left(2, 2\right)$ and recalling the properties of $\eta$ we have

\begin{align}\label{IV_b'}
|IV|\le&|h|^\alpha\int_{\Omega}\eta^2(x)\left(g_k(x)+g_k(x+h)\right)\left(\mu^2+\left|Du(x)\right|^2\right)^\frac{p-1}{2}\left|\tau_{h}Du(x)\right|dx\notag\\
\le&\varepsilon\int_{\Omega}\eta^2(x)\left(\mu^2+\left|Du(x)\right|^2+\left|Du(x+h)\right|^2\right)^{\frac{p-2}{2}}\left|\tau_hDu(x)\right|^2dx\notag\\
&+c_\varepsilon |h|^{2\alpha}\int_{B_R}\left(g_k(x)+g_k(x+h)\right)^2\left(\mu^2+\left|Du(x)\right|^2+\left|Du(x+h)\right|^2\right)^{\frac{p}{2}}dx,
\end{align}

where $2^{-k}\frac{R}{4}\le|h|\le2^{-k+1}\frac{R}{4}$ for $k\in\N.$
Using H\"{o}lder's Inequality with exponents $\left(\frac{n}{2\alpha}, \frac{n}{n-2\alpha}\right)$ and Lemma \ref{le1} we get

\begin{align}\label{IV''_b}
|IV|\le&\varepsilon\int_{\Omega}\eta^2(x)\left(\mu^2+\left|Du(x)\right|^2+\left|Du(x+h)\right|^2\right)^{\frac{p-2}{2}}\left|\tau_hDu(x)\right|^2dx\notag\\
&+c_\varepsilon |h|^{2\alpha}\left(\int_{B_R}\left(g_k(x)+g_k(x+h)\right)^{\frac{n}{\alpha}}dx\right)^\frac{2\alpha}{n}\notag\\
&\quad\cdot\left(\int_{B_R}\left(\mu^2+\left|Du(x)\right|^2+\left|Du(x+h)\right|^2\right)^{\frac{np}{2(n-2\alpha)}}dx\right)^\frac{n-2\alpha}{n}\notag\\
\le&\varepsilon\int_{\Omega}\eta^2(x)\left(\mu^2+\left|Du(x)\right|^2+\left|Du(x+h)\right|^2\right)^{\frac{p-2}{2}}\left|\tau_hDu(x)\right|^2dx\notag\\
&+c_\varepsilon |h|^{2\alpha}\left(\int_{B_t}\left(g_k(x)+g_k(x+h)\right)^{\frac{n}{\alpha}}dx\right)^\frac{2\alpha}{n}\notag\\
&\quad\cdot\left(\int_{B_R}\left(\mu^{\frac{np}{n-2\alpha}}+\left|Du(x)\right|^{\frac{np}{n-2\alpha}}\right)dx\right)^\frac{n-2\alpha}{n}.
\end{align}

Applying estimate \eqref{ByunChoOkStima} with $q=\frac{np}{n-2\alpha}$, we have 

\begin{equation}\label{BCOestimateIV_b}
\Mint_{B_{R}}\left|Du(x)\right|^\frac{np}{n-2\alpha}dx\le C\left\lbrace 1+\Mint_{B_{2R}}\left|D\psi(x)\right|^\frac{np}{n-2\alpha}dx+\left(\Mint_{B_{2R}}\left|Du(x)\right|^pdx\right)^\frac{n}{n-2\alpha}\right\rbrace.
\end{equation}

Plugging \eqref{BCOestimateIV_b} into \eqref{IV''_b} we get

\begin{align}\label{IV_b}
|IV|\le&\varepsilon\int_{\Omega}\eta^2(x)\left(\mu^2+\left|Du(x)\right|^2+\left|Du(x+h)\right|^2\right)^{\frac{p-2}{2}}\left|\tau_hDu(x)\right|^2dx\notag\\
&+c_\varepsilon |h|^{2\alpha}\left(\int_{B_R}\left(g_k(x)+g_k(x+h)\right)^{\frac{n}{\alpha}}dx\right)^\frac{2\alpha}{n}\notag\\
&\quad\cdot\left[ 1+\int_{B_{2R}}\left|D\psi(x)\right|^\frac{np}{n-2\alpha}dx+\left(\int_{B_{2R}}\left|Du(x)\right|^pdx\right)^\frac{n}{n-2\alpha}\right]^\frac{n-2\alpha}{n}.
\end{align}

In order to estimate the term $V$, we recall the properties of $\eta$, consider $2^{-k}\frac{R}{4}\le|h|\le2^{-k+1}\frac{R}{4}$ for $k\in\N$ and use \eqref{A5}, Young's Inequality with exponents $\left(2, 2\right)$, and Lemma \ref{lemma6GP}, thus getting 

\begin{align}\label{V_b'}
|V|\le&|h|^\alpha\int_{\Omega}\eta^2(x)\left(g_k(x)+g_k(x+h)\right)\left(\mu^2+\left|Du(x)\right|^2\right)^\frac{p-1}{2}\left|\tau_{h}D\psi(x)\right|dx\notag\\
\le&|h|^{\alpha}\int_{B_t}\left(g_k(x)+g_k(x+h)\right)\left(\mu^2+\left|Du(x)\right|^2\right)^\frac{p-1}{2}\left|\tau_{h}D\psi(x)\right|\notag\\
&\cdot\left(\mu^2+\left|D\psi(x)\right|^2+\left|D\psi(x+h)\right|^2\right)^\frac{p-2}{4}\cdot \left(\mu^2+\left|D\psi(x)\right|^2+\left|D\psi(x+h)\right|^2\right)^\frac{2-p}{4}dx\notag\\
\le& c|h|^{2\alpha}\int_{B_t}\left(g_k(x)+g_k(x+h)\right)^2\cdot\left(\mu^2+\left|Du(x)\right|^2\right)^{p-1}\notag\\
&\quad\cdot\left(\mu^2+\left|D\psi(x)\right|^2+\left|D\psi(x+h)\right|^2\right)^\frac{2-p}{2}dx\notag\\
&+c\int_{B_t}\left(\mu^2+\left|D\psi(x)\right|^2+\left|D\psi(x+h)\right|^2\right)^\frac{p-2}{2}\left|\tau_{h}D\psi(x)\right|^2dx\notag\\
\le& c|h|^{2\alpha}\int_{B_t}\left(g_k(x)+g_k(x+h)\right)^2\cdot\left(\mu^2+\left|Du(x)\right|^2\right)^{p-1}\notag\\
&\quad\cdot\left(\mu^2+\left|D\psi(x)\right|^2+\left|D\psi(x+h)\right|^2\right)^\frac{2-p}{2}dx+c\int_{B_R}\left|\tau_hV_p(D\psi(x))\right|^2dx.
\end{align}

By H\"{o}lder's Inequality with exponents $\left(\frac{n}{2\alpha}, \frac{n}{n-2\alpha}\right)$ and then with exponents $\left(\frac{p}{2(p-1)}, \frac{p}{2-p}\right)$, \eqref{V_b'} gives

\begin{align}\label{V_b''}
|V|\le& c|h|^{2\alpha}\left(\int_{B_t}\left(g_k(x)+g_k(x+h)\right)^\frac{n}{\alpha}dx\right)^\frac{2\alpha}{n}\notag\\
&\quad\cdot\left(\int_{B_{t}}\left(\mu^2+\left|Du(x)\right|^2\right)^{\frac{n(p-1)}{n-2\alpha}}\cdot\left(\mu^2+\left|D\psi(x)\right|^2+\left|D\psi(x+h)\right|^2\right)^{\frac{n(2-p)}{2(n-2\alpha)}}dx\right)^\frac{n-2\alpha}{n}\notag\\
&+c\int_{B_R}\left|\tau_hV_p(D\psi(x))\right|^2dx\notag\\
\le& c|h|^{2\alpha}\left(\int_{B_R}\left(g_k(x)+g_k(x+h)\right)^\frac{n}{\alpha}dx\right)^\frac{2\alpha}{n}\cdot\left(\int_{B_{R}}\left(\mu^2+\left|Du(x)\right|^2\right)^{\frac{np}{2(n-2\alpha)}}dx\right)^{\frac{2(p-1)}{p}\cdot\frac{n-2\alpha}{n}}\notag\\
&\quad\cdot\left(\int_{B_t}\left(\mu^2+\left|D\psi(x)\right|^2+\left|D\psi(x+h)\right|^2\right)^{\frac{np}{2(n-2\alpha)}}dx\right)^{\frac{2-p}{p}\cdot\frac{n-2\alpha}{n}}\notag\\
&+c\int_{B_R}\left|\tau_hV_p(D\psi(x))\right|^2dx.\notag\\
\le&c|h|^{2\alpha}\left(\int_{B_R}\left(g_k(x)+g_k(x+h)\right)^\frac{n}{\alpha}dx\right)^\frac{2\alpha}{n}\cdot\left(\int_{B_{R}}\left(\mu^{\frac{np}{n-2\alpha}}+\left|Du(x)\right|^{\frac{np}{n-2\alpha}}\right)dx\right)^{\frac{2(p-1)}{p}\cdot\frac{n-2\alpha}{n}}\notag\\
&\quad\cdot\left(\int_{B_R}\left(\mu^{\frac{np}{n-2\alpha}}+\left|D\psi(x)\right|^{\frac{np}{n-2\alpha}}\right)dx\right)^{\frac{2-p}{p}\cdot\frac{n-2\alpha}{n}}+c\int_{B_R}\left|\tau_hV_p(D\psi(x))\right|^2dx,
\end{align}

where we also used Lemma \ref{le1}. Using Young's Inequality with exponents $\left(\frac{p}{2(p-1)}, \frac{p}{2-p}\right)$, we get\\

\begin{align}\label{V_b'''}
|V|\le& c|h|^{2\alpha}\left(\int_{B_R}\left(g_k(x)+g_k(x+h)\right)^\frac{n}{\alpha}dx\right)^\frac{2\alpha}{n}\cdot\left[\left(\int_{B_{R}}\left(\mu^{\frac{np}{n-2\alpha}}+\left|Du(x)\right|^{\frac{np}{n-2\alpha}}\right)dx\right)^{\frac{n-2\alpha}{n}}\right.\notag\\
&\left.\quad+\left(\int_{B_R}\left(\mu^{\frac{np}{n-2\alpha}}+\left|D\psi(x)\right|^{\frac{np}{n-2\alpha}}\right)dx\right)^{\frac{n-2\alpha}{n}}\right]+c\int_{B_R}\left|\tau_hV_p(D\psi(x))\right|^2dx.
\end{align}

Using \eqref{BCOestimateIV_b} to estimate the second integral of the right-hand side of \eqref{V_b'''}, we get

\begin{align}\label{V_b}
|V|\le& c|h|^{2\alpha}\left(\int_{B_R}\left(g_k(x)+g_k(x+h)\right)^\frac{n}{\alpha}dx\right)^\frac{2\alpha}{n}\notag\\
&\quad\cdot\left[1+\int_{B_{2R}}\left|D\psi(x)\right|^\frac{np}{n-2\alpha}dx+\left(\int_{B_{2R}}\left|Du(x)\right|^pdx\right)^\frac{n}{n-2\alpha}\right]^{\frac{n-2\alpha}{n}}\notag\\
&+c\int_{B_R}\left|\tau_hV_p(D\psi(x))\right|^2dx.
\end{align}

Now we consider the term $VI$. Recalling \eqref{A5}, taking $2^{-k}\frac{R}{4}\le|h|\le2^{-k+1}\frac{R}{4}$ for $k\in\N$, the properties of $\eta$ and using H\"{o}lder's Inequality with exponents $\left(\frac{n}{2\alpha}, \frac{n}{n-2\alpha}\right)$, and $\left(p,\frac{p}{p-1}\right)$ we get

\begin{align}\label{VI_b'}
|VI|\le&|h|^\alpha\int_{\Omega}\eta(x)\left|D\eta(x)\right|\left(g_k(x)+g_k(x+h)\right)\left(\mu^2+\left|Du(x)\right|^2\right)^\frac{p-1}{2}\left|\tau_{h}(u-\psi)(x)\right|dx\notag\\
\le& \frac{c|h|^\alpha}{R}\left(\int_{B_R}\left(g_k(x)+g_k(x+h)\right)^\frac{n}{2\alpha}dx\right)^\frac{2\alpha}{n}\notag\\
&\cdot\left(\int_{B_R}\left(\mu^2+\left|Du(x)\right|^2\right)^\frac{n(p-1)}{2(n-2\alpha)}\left|\tau_{h}\left(u-\psi(x)\right)\right|^\frac{n}{n-2\alpha}dx\right)^\frac{n-2\alpha}{n}\notag\\
\le&\frac{c|h|^\alpha}{R}\left(\int_{B_R}\left(g_k(x)+g_k(x+h)\right)^\frac{n}{2\alpha}dx\right)^\frac{2\alpha}{n}\cdot\left[\left(\int_{B_R}\left(\mu^2+\left|Du(x)\right|^2\right)^\frac{np}{2(n-2\alpha)}dx\right)^{\frac{p-1}{p}\cdot\frac{n-2\alpha}{n}}\right.\notag\\
&\left.\quad\cdot\left(\int_{B_R}\left|\tau_{h}(u-\psi)(x)\right|^\frac{np}{n-2\alpha}dx\right)^\frac{n-2\alpha}{np}\right].
\end{align}

By virtue of Lemma \ref{le1} we have

\begin{align}\label{VI_b*}
|VI|\le& \frac{c|h|^{\alpha+1}}{R}\left(\int_{B_R}\left(g_k(x)+g_k(x+h)\right)^\frac{n}{2\alpha}dx\right)^\frac{2\alpha}{n}\cdot\left[\left(\int_{B_{R}}\left(\mu^2+\left|Du(x)\right|^2\right)^\frac{np}{2(n-2\alpha)}dx\right)^{\frac{p-1}{p}\cdot\frac{n-2\alpha}{n}}\right.\notag\\
&\left.\cdot\left(\int_{B_{2R}}\left|D(u-\psi)(x)\right|^\frac{np}{n-2\alpha}dx\right)^\frac{n-2\alpha}{np}\right]\notag\\
\le& \frac{c|h|^{\alpha+1}}{R^{1-\alpha}}\left(\int_{B_R}\left(g_k(x)+g_k(x+h)\right)^\frac{n}{\alpha}dx\right)^\frac{\alpha}{n}\cdot\left[\left(\int_{B_R}\left(\mu^\frac{np}{n-2\alpha}+\left|Du(x)\right|^\frac{np}{n-2\alpha}\right)dx\right)^{\frac{p-1}{p}\cdot\frac{n-2\alpha}{n}}\right.\notag\\
&\left.\cdot\left(\int_{B_{\lambda R}}\left|D(u-\psi)(x)\right|^\frac{np}{n-2\alpha}dx\right)^\frac{n-2\alpha}{np}\right],
\end{align}

where we used the fact that $\LBrace g_k\RBrace_k\subset L^\frac{n}{\alpha}(\Omega)\subset L^\frac{n}{2\alpha}(\Omega)$, with the following estimate

\[
\lnorm g_k\rnorm_{L^\frac{n}{2\alpha}\left(B_R\right)}\le cR^\alpha\lnorm g_k\rnorm_{L^\frac{n}{\alpha}\left(B_R\right)}.
\]

Now, by Young's Inequality with exponents $\left(p, \frac{p}{p-1}\right)$ and \eqref{BCOestimateIV_b}, \eqref{VI_b*} becomes

\begin{align}\label{VI_b}
|VI|\le& \frac{c|h|^{\alpha+1}}{R^{1-\alpha}}\left(\int_{B_R}\left(g_k(x)+g_k(x+h)\right)^\frac{n}{\alpha}dx\right)^\frac{\alpha}{n}\notag\\
&\cdot\left[1+\int_{B_{2\lambda R}}\left|D\psi(x)\right|^\frac{np}{n-2\alpha}+\left(\int_{B_{2\lambda R}}\left|Du(x)\right|^pdx\right)^\frac{n}{n-2\alpha}\right]^\frac{n-2\alpha}{n}.
\end{align}

Plugging \eqref{I_b}, \eqref{II_b}, \eqref{III_b}, \eqref{IV_b}, \eqref{V_b} and \eqref{VI_b} into \eqref{differenceinequality}, recalling the properties of $\eta$ and choosing $\varepsilon=\frac{\nu}{6}$, and using Lemma \ref{lemma6GP} to estimate the left-hand side, we get

\begin{align}\label{full_b'}
&\int_{B_{\frac{R}{2}}}\left|\tau_{h}V_p\left(Du(x)\right)\right|^2dx\le cR^{2\alpha}\left[\int_{B_{R}}\left(1+\left|D\psi(x)\right|^\frac{np}{n-2\alpha}\right)dx\right]^\frac{n-2\alpha}{n}\notag\\
&+\frac{c|h|^p}{R^{p-2\alpha}}\left(\int_{B_{R}}\left|D(u-\psi)(x)\right|^\frac{np}{n-2\alpha}dx\right)^\frac{n-2\alpha}{n}+c|h|^{2\alpha}\left(\int_{B_R}\left(g_k(x)+g_k(x+h)\right)^{\frac{n}{\alpha}}dx\right)^\frac{2\alpha}{n}\notag\\
&\quad\cdot\left[ 1+\int_{B_{2R}}\left|D\psi(x)\right|^\frac{np}{n-2\alpha}dx+\left(\int_{B_{2R}}\left|Du(x)\right|^pdx\right)^\frac{n}{n-2\alpha}\right]^\frac{n-2\alpha}{n}\notag\\
&+c\int_{B_R}\left|\tau_hV_p(D\psi(x))\right|^2dx+\frac{c|h|^{\alpha+1}}{R^{1-\alpha}}\left(\int_{B_R}\left(g_k(x)+g_k(x+h)\right)^\frac{n}{\alpha}dx\right)^\frac{\alpha}{n}\notag\\
&\quad\cdot\left[ 1+\int_{B_{2\lambda R}}\left|D\psi(x)\right|^\frac{np}{n-2\alpha}dx+\left(\int_{B_{2\lambda R}}\left|Du(x)\right|^pdx\right)^\frac{n}{n-2\alpha}\right]^\frac{n-2\alpha}{n}.
\end{align}

Now, for $\beta\in(0, 1)$, for suffinciently small values of $|h|$ that will be made clear later, we can choose $R=\lAbs h\rAbs^\beta$, so \eqref{full_b'} becomes

\begin{align}\label{full_b_}
&\int_{B_{\frac{R}{2}}}\left|\tau_{h}V_p\left(Du(x)\right)\right|^2dx\le c|h|^{2\alpha\beta}\left[\int_{B_{R}}\left(1+\left|D\psi(x)\right|^\frac{np}{n-2\alpha}\right)dx\right]^\frac{n-2\alpha}{n}\notag\\
&+c|h|^{p(1-\beta)+2\alpha\beta}\left(\int_{B_{R}}\left|D(u-\psi)(x)\right|^\frac{np}{n-2\alpha}dx\right)^\frac{n-2\alpha}{n}+c|h|^{2\alpha}\left(\int_{B_R}\left(g_k(x)+g_k(x+h)\right)^{\frac{n}{\alpha}}dx\right)^\frac{2\alpha}{n}\notag\\
&\quad\cdot\left[ 1+\int_{B_{2R}}\left|D\psi(x)\right|^\frac{np}{n-2\alpha}dx+\left(\int_{B_{2R}}\left|Du(x)\right|^pdx\right)^\frac{n}{n-2\alpha}\right]^\frac{n-2\alpha}{n}\notag\\
&+c\int_{B_R}\left|\tau_hV_p(D\psi(x))\right|^2dx+c|h|^{\alpha-\beta+\alpha\beta+1}\left(\int_{B_R}\left(g_k(x)+g_k(x+h)\right)^\frac{n}{\alpha}dx\right)^\frac{\alpha}{n}\notag\\
&\quad\cdot \left[ 1+\int_{B_{2\lambda R}}\left|D\psi(x)\right|^\frac{np}{n-2\alpha}dx+\left(\int_{B_{2\lambda R}}\left|Du(x)\right|^pdx\right)^\frac{n}{n-2\alpha}\right]^\frac{n-2\alpha}{n}.
\end{align}

Now, since $\alpha, \beta\in(0, 1)$, if we set

\begin{align*}\label{piDefinition}
	&p_1=2\alpha\beta\in(0, 2),\qquad p_2=p(1-\beta)+2\alpha\beta\in(0, 4),\notag\\&p_3=2\alpha\in(0, 2),\qquad p_4=\alpha-\beta+\alpha\beta+1=(\alpha+1)(1-\beta)+2\alpha\beta\in(0, 3)
\end{align*}

we have

\begin{equation*}
\min_{i\in\set{1, 2, 3, 4}} p_i=p_1=2\alpha\beta.
\end{equation*}

Now let us divide both sides of \eqref{full_b_} by $|h|^{2\alpha\beta}$.\\

\begin{align}\label{full_b''}
\int_{B_{\frac{R}{2}}}&\frac{\left|\tau_{h}V_p\left(Du(x)\right)\right|^2}{|h|^{2{\alpha\beta}}}dx\le c\left[\int_{B_{R}}\left(1+\left|D\psi(x)\right|^\frac{np}{n-2\alpha}\right)dx\right]^\frac{n-2\alpha}{n}\notag\\
&+c|h|^{p(1-\beta)}\left(\int_{B_{R}}\left|D(u-\psi)(x)\right|^\frac{np}{n-2\alpha}dx\right)^\frac{n-2\alpha}{n}\notag\\
&+c|h|^{2\alpha(1-\beta)}\left(\int_{B_R}\left(g_k(x)+g_k(x+h)\right)^{\frac{n}{\alpha}}dx\right)^\frac{2\alpha}{n}\notag\\
&\quad\cdot\left[ 1+\int_{B_{2R}}\left|D\psi(x)\right|^\frac{np}{n-2\alpha}dx+\left(\int_{B_{2R}}\left|Du(x)\right|^pdx\right)^\frac{n}{n-2\alpha}\right]^\frac{n-2\alpha}{n}\notag\\
&+c\int_{B_R}\frac{\left|\tau_hV_p(D\psi(x))\right|^2}{|h|^{2\alpha}}dx\notag\\
&+c|h|^{(\alpha+1)(1-\beta)}\left(\int_{B_R}\left(g_k(x)+g_k(x+h)\right)^\frac{n}{\alpha}dx\right)^\frac{\alpha}{n}\notag\\
&\quad\cdot \left[ 1+\int_{B_{2\lambda R}}\left|D\psi(x)\right|^\frac{np}{n-2\alpha}dx+\left(\int_{B_{2\lambda R}}\left|Du(x)\right|^pdx\right)^\frac{n}{n-2\alpha}\right]^\frac{n-2\alpha}{n},
\end{align}

where we also used the fact that, if $|h|<1$ and $\beta, \alpha\in(0,1)$, then $|h|^{-2\alpha\beta}\le|h|^{-2\alpha}$.\\
In order to conclude, we have to take the $L^q$ norm with the measure $\frac{dh}{|h|^n}$ restricted to the ball $B\left(0, \frac{R}{4}\right)$ on the $h-$space, of the $L^2$ norm of the difference quotient of order $\alpha\beta$ of the function $V_p(Du)$. Since we have to integrate with respect to the measure $\frac{dh}{|h|^n}$ on the ball $B\left(0, \frac{R}{4}\right)$ and, for each $k\in \N$, the integral in the second-last line of \eqref{full_b''} is taken for $2^{-k}\frac{R}{4}\le|h|\le2^{-k+1}\frac{R}{4}$, it is useful to notice what follows

\begin{equation*}
B\left(0, \frac{R}{4}\right)=\bigcup_{k=1}^{\infty}\left(B\left(0, 2^{-k+1}\frac{R}{4}\right)\setminus B\left(0, 2^{-k}\frac{R}{4}\right)\right)=:\bigcup_{k=1}^\infty E_k,
\end{equation*}

and it is also worth noticing that the choice of the radius $R=\lAbs h\rAbs^\beta$ is possible for small values of $|h|$, since, for $k\in\N$, $2^{-k}\frac{R}{4}\le|h|\le2^{-k+1}\frac{R}{4}$ if and only if $2^{-\frac{k+2}{1-\beta}}\le|h|\le2^{-\frac{k+1}{1-\beta}}$.\\ 
We obtain the following estimate

\begin{align}\label{full_b*}
&\int_{B_{\frac{R}{4}(0)}}\left(\int_{B_{\frac{R}{2}}}\frac{\left|\tau_{h}V_p\left(Du(x)\right)\right|^2}{|h|^{2\alpha\beta}}dx\right)^\frac{q}{2}\frac{dh}{|h|^n}\le c\int_{B_{\frac{R}{4}(0)}}\left[\int_{B_{R}}\left(1+\left|D\psi(x)\right|^\frac{np}{n-2\alpha}\right)dx\right]^\frac{q(n-2\alpha)}{2n}\frac{dh}{|h|^n}\notag\\
&+c\int_{B_{\frac{R}{4}(0)}}|h|^\frac{qp(1-\beta)}{2}\frac{dh}{|h|^n}\cdot\left(\int_{B_{R}}\left|D(u-\psi)(x)\right|^\frac{np}{n-2\alpha}dx\right)^\frac{q(n-2\alpha)}{2n}\notag\\
&+c\sum_{k=1}^\infty\int_{E_k}\left|h\right|^{q\alpha(1-\beta)}\left(\int_{B_R}\left(g_k(x)+g_k(x+h)\right)^{\frac{n}{\alpha}}dx\right)^\frac{q\alpha}{n}\frac{dh}{|h|^n}\notag\\
&\quad\cdot\left[ 1+\int_{B_{2R}}\left|D\psi(x)\right|^\frac{np}{n-2\alpha}dx+\left(\int_{B_{2R}}\left|Du(x)\right|^pdx\right)^\frac{n}{n-2\alpha}\right]^\frac{q(n-2\alpha)}{2n}\notag\\
&+c\int_{B_{\frac{R}{4}(0)}}\left(\int_{B_R}\frac{\left|\tau_hV_p(D\psi(x))\right|^2}{|h|^{2\alpha}}dx\right)^\frac{q}{2}\frac{dh}{|h|^n}\notag\\
&+c\sum_{k=1}^\infty\int_{E_k}\left|h\right|^\frac{q(\alpha+1)(1-\beta)}{2}\left(\int_{B_R}\left(g_k(x)+g_k(x+h)\right)^\frac{n}{\alpha}dx\right)^\frac{q\alpha}{2n}\frac{dh}{|h|^n}\notag\\
&\quad\cdot \left[ 1+\int_{B_{2\lambda R}}\left|D\psi(x)\right|^\frac{np}{n-2\alpha}dx+\left(\int_{B_{2\lambda R}}\left|Du(x)\right|^pdx\right)^\frac{n}{n-2\alpha}\right]^\frac{q(n-2\alpha)}{2n}.
\end{align}

Now, in order to simplify the notations, we set

\begin{equation}\label{N}
\tilde{N}=\int_{B_{2\lambda R}}\left(1+\left|Du(x)\right|^p+\left|Du(x)\right|^{\frac{np}{n-2\alpha}}+\left|D\psi(x)\right|^p+\left|D\psi(x)\right|^{\frac{np}{n-2\alpha}}\right)dx,
\end{equation}

and write \eqref{full_b*} as follows

\begin{align}\label{full_b**}
&\int_{B_{\frac{R}{4}(0)}}\left(\int_{B_{\frac{R}{2}}}\frac{\left|\tau_{h}V_p\left(Du(x)\right)\right|^2}{|h|^{2\alpha\beta}}dx\right)^\frac{q}{2}\frac{dh}{|h|^n}\le C\int_{B_{\frac{R}{4}(0)}}|h|^\frac{qp(1-\beta)}{2}\frac{dh}{|h|^n}\notag\\
&+C\sum_{k=1}^\infty\int_{E_k}\left|h\right|^{q\alpha(1-\beta)}\left(\int_{B_R}\left(g_k(x)+g_k(x+h)\right)^{\frac{n}{\alpha}}dx\right)^\frac{q\alpha}{n}\frac{dh}{|h|^n}\notag\\
&+C\int_{B_{\frac{R}{4}(0)}}\left(\int_{B_R}\frac{\left|\tau_hV_p(D\psi(x))\right|^2}{|h|^{2\alpha}}dx\right)^\frac{q}{2}\frac{dh}{|h|^n}\notag\\
&+C\sum_{k=1}^\infty\int_{E_k}\left|h\right|^\frac{q(\alpha+1)(1-\beta)}{2}\left(\int_{B_R}\left(g_k(x)+g_k(x+h)\right)^\frac{n}{\alpha}dx\right)^\frac{q\alpha}{2n}\frac{dh}{|h|^n},
\end{align}

where the constant $C$ now depends on $\nu, \ell, L, n, p, q, \alpha, R, \tilde{N}.$\\
Applying Young's Inequality with exponents $\left(2, 2\right)$ to the second and the fourth integral of the right-hand side of \eqref{full_b**}, we get

\begin{align}\label{full_b****}
&\int_{B_{\frac{R}{4}(0)}}\left(\int_{B_{\frac{R}{2}}}\frac{\left|\tau_{h}V_p\left(Du(x)\right)\right|^2}{|h|^{2\alpha\beta}}dx\right)^\frac{q}{2}\frac{dh}{|h|^n}\le C\int_{B_{\frac{R}{4}(0)}}|h|^\frac{qp(1-\beta)}{2}\frac{dh}{|h|^n}\notag\\
&+C\int_{B_{\frac{R}{4}(0)}}\left|h\right|^{2q\alpha(1-\beta)}\frac{dh}{|h|^n}+C\sum_{k=1}^\infty\int_{E_k}\left(\int_{B_R}\left(g_k(x)+g_k(x+h)\right)^{\frac{n}{\alpha}}dx\right)^\frac{2q\alpha}{n}\frac{dh}{|h|^n}\notag\\
&+C\int_{B_{\frac{R}{4}(0)}}\left|h\right|^{q(\alpha+1)(1-\beta)}\frac{dh}{|h|^n}+C\sum_{k=1}^\infty\int_{E_k}\left(\int_{B_R}\left(g_k(x)+g_k(x+h)\right)^\frac{n}{\alpha}dx\right)^\frac{q\alpha}{n}\frac{dh}{|h|^n}\notag\\
&+C\int_{B_{\frac{R}{4}(0)}}\left(\int_{B_R}\frac{\left|\tau_hV_p(D\psi(x))\right|^2}{|h|^{2\alpha}}dx\right)^\frac{q}{2}\frac{dh}{|h|^n}.
\end{align}

Now let us observe that, since $\alpha, \beta\in (0, 1)$ and $1<p<2$, if we set $p_1=\frac{p(1-\beta)}{2}$, $p_2=2\alpha(1-\beta)$ and $p_3=(\alpha+1)(1-\beta)$ and, for each $i=1, 2, 3$,  $q_i=q\cdot p_i$, we have

$$\kappa:=\min_{i\in\set{1, 2, 3}}q_i>0$$

and since $|h|<1$ we can write \eqref{full_b****} as follows

\begin{align}\label{full_b*****}
&\int_{B_{\frac{R}{4}(0)}}\left(\int_{B_{\frac{R}{2}}}\frac{\left|\tau_{h}V_p\left(Du(x)\right)\right|^2}{|h|^{2\beta}}dx\right)^\frac{q}{2}\frac{dh}{|h|^n}\le
\notag\\&C\sum_{k=1}^\infty\int_{E_k}\left(\int_{B_R}\left(g_k(x)+g_k(x+h)\right)^{\frac{n}{\alpha}}dx\right)^\frac{2q\alpha}{n}\frac{dh}{|h|^n}\notag\\
&+C\sum_{k=1}^\infty\int_{E_k}\left(\int_{B_R}\left(g_k(x)+g_k(x+h)\right)^\frac{n}{\alpha}dx\right)^\frac{q\alpha}{n}\frac{dh}{|h|^n}\notag\\
&+C\int_{B_{\frac{R}{4}(0)}}\left(\int_{B_R}\frac{\left|\tau_hV_p(D\psi(x))\right|^2}{|h|^{2\alpha}}dx\right)^\frac{q}{2}\frac{dh}{|h|^n}\notag\\
&+C\int_{B_{\frac{R}{4}(0)}}\left|h\right|^\kappa\frac{dh}{|h|^n}=I_1+I_2+I_3+I_4
\end{align}

Now we notice that 

\begin{equation}\label{I_3}
I_3\le\left\Arrowvert V_p\left(D\psi\right)\right\Arrowvert_{B^\alpha_{2,q}(B_R)},
\end{equation}

which is finite by hypothesis.\\
For what concerns the term $I_4$, by calculating it in polar coordinates, we get

\begin{equation}\label{I_4}
I_4=C\int_0^\frac{R}{4}\rho^{\kappa-1}d\rho=C(n, p, q, \alpha, R),
\end{equation}

since $\kappa>0.$\\
Now let us write the integral $I_1$ in polar coordinates, so $h\in E_k$ if and only if $h=\rho\xi$ for $2^{-k}\frac{R}{4}\le \rho<2^{-k+1}\frac{R}{4}$ and some $\xi$ in the unit sphere $S^{n-1}$ on $\R^n.$ Denoting by $d\sigma(\xi)$ the surface measure on $S^{n-1}$, we have

\begin{align}\label{I_1'}
I_1=&C\sum_{k=1}^\infty\int_{r_{k}}^{r_{k-1}}\int_{S^{n-1}}\left(\int_{B_R}\left(g_k(x+\rho\xi)-g_k(x)\right)^\frac{n}{\alpha}\right)^\frac{2\alpha q}{n}d\sigma(\xi)\frac{d\rho}{\rho}\notag\\
\le& C\sum_{k=1}^\infty\int_{r_{k}}^{r_{k-1}}\int_{S^{n-1}}\left\Arrowvert \tau_{\rho\xi}g_k+g_k\right\Arrowvert_{L^\frac{n}{\alpha}(B_R)}^{2q} d\sigma(\xi)\frac{d\rho}{\rho},
\end{align}

where we set $r_k=2^{-k}\frac{R}{4}$. Let us note that, for each $\xi\in S^{n-1}$ and $r_k\le\rho\le r_{k-1},$

\begin{equation}\label{I_1''}
\left\Arrowvert \tau_{\rho\xi}g_k+g_k\right\Arrowvert_{L^\frac{n}{\alpha}(B_R)}\le\left\Arrowvert g_k\right\Arrowvert_{L^\frac{n}{\alpha}(B_{R-r_k\xi})}+\left\Arrowvert g_k\right\Arrowvert_{L^\frac{n}{\alpha}(B_R)}\le 2\left\Arrowvert g_k\right\Arrowvert_{L^\frac{n}{\alpha}\left(B_{R+\frac{R}{4}}\right)}.
\end{equation}

So, recalling the continuous embedding $\ell^q\left(L^\frac{n}{\alpha}\left(B_{2R}\right)\right)\subset\ell^{2q}\left(L^\frac{n}{\alpha}\left(B_{2R}\right)\right)$, by \eqref{I_1'} and \eqref{I_1''}, we get

\begin{equation}\label{I_1}
I_1\le C\left\Arrowvert\left\lbrace g_k\right\rbrace_k\right\Arrowvert_{\ell^q\left(L^\frac{n}{\alpha}\left(B_{2R}\right)\right)}^{2q}.
\end{equation}

We can argue in a similiar way to estimate the term $I_2$, thus getting

\begin{equation}\label{I_2}
I_2\le C\left\Arrowvert\left\lbrace g_k\right\rbrace_k\right\Arrowvert_{\ell^q\left(L^\frac{n}{\alpha}\left(B_{2R}\right)\right)}^q.
\end{equation}

Inserting \eqref{I_1}, \eqref{I_2}, \eqref{I_3}, and \eqref{I_4} in \eqref{full_b*****}, we have

\begin{align}\label{conclusion2}
\left\Arrowvert\frac{\tau_h V_p\left(Du\right)}{|h|^{\alpha\beta}}\right\Arrowvert_{L^q\left(\frac{dh}{|h|^n}; L^2\left(B_{\frac{R}{2}}\right)\right)}\le& C\left(1+\left\Arrowvert V_p\left(D\psi\right)\right\Arrowvert_{B^\alpha_{2,q}(B_R)}+\left\Arrowvert\left\lbrace g_k\right\rbrace_k\right\Arrowvert_{\ell^q\left(L^\frac{n}{\alpha}\left(B_{2R}\right)\right)}^{2q}\right).
\end{align}

Recalling explicitly the dependence of the constant $C$ on the value of $\tilde{N}$ given by \eqref{N}, for a suitable exponent $\sigma=\sigma(n, p, q, \alpha)>0$, using the fact that $\frac{np}{n-2\alpha}>p$ recalling \eqref{ByunChoOkStima} and  using Lemma \ref{lemma2.4EP} and Lemma \ref{EP2.2}, we can conclude with the estimate

\begin{align}\label{estimate2pf}
\left\Arrowvert\frac{\tau_h V_p\left(Du\right)}{|h|^{\alpha\beta}}\right\Arrowvert_{L^q\left(\frac{dh}{|h|^n}; L^2\left(B_{\frac{R}{2}}\right)\right)}\le C&\Bigg(1+\left\Arrowvert Du\right\Arrowvert_{L^p(B_{4R})}+\left\Arrowvert V_p\left(D\psi\right)\right\Arrowvert_{B^\alpha_{2,q}(B_{4R})}\notag\\
&+\left\Arrowvert\left\lbrace g_k\right\rbrace_k\right\Arrowvert_{\ell^q\left(L^\frac{n}{\alpha}\left(B_{2R}\right)\right)}\Bigg)^\sigma,
\end{align}

that is the conclusion.
\end{proof}

\section{Proof of Theorem \ref{thm3}}\label{Thm3Pf}

This section is devoted to the proof of Theorem \ref{thm3}, which is obtained using the same arguments of the previous section, taking into account that, here, the assumption \ref{A5} is replaced by the assumption \ref{A6}.

\begin{proof}
	Since, by the hypothesis, $V_p(D\psi)\in B^\gamma_{2,\infty,\mathrm{loc}}(\Omega)$ with $\alpha<\gamma<1$ then, recalling the definition \eqref{B-LInfsemiorm} and  using Lemma \ref{Lemma2.9EP}, we have $V_p(D\psi)\in L^{\frac{2n}{n-2\alpha}}_{\mathrm{loc}}(\Omega)$, and so $D\psi\in L^{\frac{np}{n-2\alpha}}_{\mathrm{loc}}(\Omega)$.
	This proof goes exactly like the one of Theorem \ref{thm2} until we arrive at the estimate \eqref{differenceinequality}, and the terms $I$, $II$ and $III$ can be treated in the same way, using \eqref{I_b}, \eqref{II_b} and \eqref{III_b} respectively. We just need to use the assumption \eqref{A6} instead of \eqref{A5}, in order to estimate the terms $IV$, $V$ and $VI$.\\
	For what concerns the term $IV$, using the assumption \eqref{A6}, Young's Inequality with exponents $\left(2, 2\right)$, H\"{o}lder's Inequality with exponents $\left(\frac{n}{2\alpha}, \frac{n}{n-2\alpha}\right)$, and Lemma \ref{le1} we get, for $|h|<\frac{R}{4}$,
	
	\begin{align}\label{IV'_c}
	|IV|&\le\varepsilon\int_{\Omega}\eta^2(x)\left(\mu^2+\left|Du(x)\right|^2+\left|Du(x+h)\right|^2\right)^{\frac{p-2}{2}}\left|\tau_hDu(x)\right|^2dx\notag\\
	&\quad+c_\varepsilon |h|^{2\alpha}\left(\int_{B_t}\left(g(x)+g(x+h)\right)^{\frac{n}{\alpha}}dx\right)^\frac{2\alpha}{n}\notag\\
	&\quad\quad\cdot\left(\int_{B_R}\left(\mu^{\frac{np}{n-2\alpha}}+\left|Du(x)\right|^{\frac{np}{n-2\alpha}}\right)dx\right)^\frac{n-2\alpha}{n}.
	\end{align}
	
	and using \eqref{BCOestimateIV_b}, we obtain

\begin{align}\label{IV_c}
|IV|&\le\varepsilon\int_{\Omega}\eta^2(x)\left(\mu^2+\left|Du(x)\right|^2+\left|Du(x+h)\right|^2\right)^{\frac{p-2}{2}}\left|\tau_hDu(x)\right|^2dx\notag\\
&\quad+c_\varepsilon |h|^{2\alpha}\left(\int_{B_t}\left(g(x)+g(x+h)\right)^{\frac{n}{\alpha}}dx\right)^\frac{2\alpha}{n}\notag\\
&\quad\quad\cdot\left[ 1+\int_{B_{2R}}\left|D\psi(x)\right|^\frac{np}{n-2\alpha}dx+\left(\int_{B_{2R}}\left|Du(x)\right|^pdx\right)^\frac{n}{n-2\alpha}\right]^\frac{n-2\alpha}{n}.
\end{align}

Let us consider, now, the term $V$ to which we apply the assumption \eqref{A6} in place of \eqref{A5}, and by the same arguments that we used in the previous section in order to obtain \eqref{V_b}, we have, for all $h\in B_{\frac{R}{4}}(0)$,

\begin{align}\label{V_c}
|V|&\le c|h|^{2\alpha}\left(\int_{B_R}\left(g(x)+g(x+h)\right)^\frac{n}{\alpha}dx\right)^\frac{2\alpha}{n}\notag\\
&\quad\quad\cdot\left[1+\int_{B_{2R}}\left|D\psi(x)\right|^\frac{np}{n-2\alpha}dx+\left(\int_{B_{2R}}\left|Du(x)\right|^pdx\right)^\frac{n}{n-2\alpha}\right]^{\frac{n-2\alpha}{n}}\notag\\
&\quad+c\int_{B_R}\left|\tau_hV_p(D\psi(x))\right|^2dx.
\end{align}

For what concerns the term $VI$, again, using the assumption \eqref{A6}, and the same arguments we used in the previous section in order to get \eqref{VI_b}, for $|h|<\frac{R}{4}$, we obtain

\begin{align}\label{VI_c}
|VI|\le& \frac{c|h|^{\alpha+1}}{R^{1-\alpha}}\left(\int_{B_R}\left(g(x)+g(x+h)\right)^\frac{n}{\alpha}dx\right)^\frac{\alpha}{n}\notag\\
&\cdot\left[1+\int_{B_{2\lambda R}}\left|D\psi(x)\right|^\frac{np}{n-2\alpha}+\left(\int_{B_{2\lambda R}}\left|Du(x)\right|^pdx\right)^\frac{n}{n-2\alpha}\right]^\frac{n-2\alpha}{n}.
\end{align}

Now we plug \eqref{I_b}, \eqref{II_b}, \eqref{III_b}, \eqref{IV_c}, \eqref{V_c} and \eqref{VI_c} into \eqref{differenceinequality}, choose $\varepsilon=\frac{\nu}{6}$, recall the properties of $\eta$ and use Lemma \ref{lemma6GP} and Lemma \ref{le1}, thus getting

\begin{align}\label{full_c}
&\int_{B_{\frac{R}{2}}}\left|\tau_{h}V_p\left(Du(x)\right)\right|^2dx\le cR^{2\alpha}\left[\int_{B_{R}}\left(1+\left|D\psi(x)\right|^\frac{np}{n-2\alpha}\right)dx\right]^\frac{n-2\alpha}{n}\notag\\
&\quad+\frac{c|h|^p}{R^{p-2\alpha}}\left(\int_{B_{R}}\left|D(u-\psi)(x)\right|^\frac{np}{n-2\alpha}dx\right)^\frac{n-2\alpha}{n}+c|h|^{2\alpha}\left(\int_{B_R}\left(g(x)+g(x+h)\right)^{\frac{n}{\alpha}}dx\right)^\frac{2\alpha}{n}\notag\\
&\quad\quad\cdot\left[ 1+\int_{B_{2R}}\left|D\psi(x)\right|^\frac{np}{n-2\alpha}dx+\left(\int_{B_{2R}}\left|Du(x)\right|^pdx\right)^\frac{n}{n-2\alpha}\right]^\frac{n-2\alpha}{n}\notag\\
&\quad+c\int_{B_R}\left|\tau_hV_p(D\psi(x))\right|^2dx+\frac{c|h|^{\alpha+1}}{R^{1-\alpha}}\left(\int_{B_R}\left(g(x)+g(x+h)\right)^\frac{n}{\alpha}dx\right)^\frac{\alpha}{n}\notag\\
&\quad\quad\cdot \left[ 1+\int_{B_{2\lambda R}}\left|D\psi(x)\right|^\frac{np}{n-2\alpha}dx+\left(\int_{B_{2\lambda R}}\left|Du(x)\right|^pdx\right)^\frac{n}{n-2\alpha}\right]^\frac{n-2\alpha}{n}.
\end{align}

Now let us notice that, since, for any $\beta\in(0, 1)$, $|h|\le\frac{|h|^\beta}{4}$ if and only if $|h|\le2^{- \frac{2}{1-\beta}}$, we can choose $R=|h|^{\beta}$ and divide both sides of \eqref{full_c} by $|h|^{2\alpha\beta}$, so we get \\

\begin{align}\label{full_c'}
&\int_{B_{\frac{R}{2}}}\frac{\left|\tau_{h}V_p\left(Du(x)\right)\right|^2}{|h|^{2\alpha\beta}}dx\le c\left[\int_{B_{R}}\left(1+\left|D\psi(x)\right|^\frac{np}{n-2\alpha}\right)dx\right]^\frac{n-2\alpha}{n}\notag\\
&\quad+c|h|^{p(1-\beta)}\left(\int_{B_{R}}\left|D(u-\psi)(x)\right|^\frac{np}{n-2\alpha}dx\right)^\frac{n-2\alpha}{n}+c|h|^{2\alpha(1-\beta)}\left(\int_{B_{2R}}g^{\frac{n}{\alpha}}(x)dx\right)^\frac{2\alpha}{n}\notag\\
&\quad\quad\cdot\left[ 1+\int_{B_{2R}}\left|D\psi(x)\right|^\frac{np}{n-2\alpha}dx+\left(\int_{B_{2R}}\left|Du(x)\right|^pdx\right)^\frac{n}{n-2\alpha}\right]^\frac{n-2\alpha}{n}\notag\\
&\quad+c\int_{B_R}\frac{\left|\tau_hV_p(D\psi(x))\right|^2}{|h|^{2\alpha}}dx+c|h|^{(1-\beta)(\alpha+1)}\left(\int_{B_R}g^\frac{n}{\alpha}(x)dx\right)^\frac{\alpha}{n}\notag\\
&\quad\quad\cdot \left[ 1+\int_{B_{2\lambda R}}\left|D\psi(x)\right|^\frac{np}{n-2\alpha}dx+\left(\int_{B_{2\lambda R}}\left|Du(x)\right|^pdx\right)^\frac{n}{n-2\alpha}\right]^\frac{n-2\alpha}{n},
\end{align}

where we also used Lemma \ref{le1} and the fact that, for $|h|<\frac{R}{4}<R<1$, since $\alpha, \beta\in(0,1)$, $|h|^{-2\alpha\beta}<|h|^{-2\alpha}$.\\
Using Young's Inequality with exponents $\left(\frac{n}{2\alpha}, \frac{n}{n-2\alpha}\right)$, \eqref{full_c'} becomes

\begin{align}\label{full_c'''}
&\int_{B_{\frac{R}{2}}}\frac{\left|\tau_{h}V_p\left(Du(x)\right)\right|^2}{|h|^{2\alpha\beta}}dx\le c\left[\int_{B_{R}}\left(1+\left|D\psi(x)\right|^\frac{np}{n-2\alpha}\right)dx\right]^\frac{n-2\alpha}{n}\notag\\
&\quad+c|h|^{p(1-\beta)}\left(\int_{B_{R}}\left|D(u-\psi)(x)\right|^\frac{np}{n-2\alpha}dx\right)^\frac{n-2\alpha}{n}+c|h|^{2\alpha(1-\beta)}\left[\int_{B_{2R}}g^{\frac{n}{\alpha}}(x)dx\right.\notag\\
&\left.\quad\quad+ 1+\int_{B_{2R}}\left|D\psi(x)\right|^\frac{np}{n-2\alpha}dx+\left(\int_{B_{2R}}\left|Du(x)\right|^pdx\right)^\frac{n}{n-2\alpha}\right]\notag\\
&\quad+c\int_{B_R}\frac{\left|\tau_hV_p(D\psi(x))\right|^2}{|h|^{2\alpha}}dx+c|h|^{(1-\beta)(\alpha+1)}\left[\left(\int_{B_{2R}}g^\frac{n}{\alpha}(x)dx\right)^{\frac{1}{2}}\right.\notag\\
&\left.\quad\quad+ 1+\int_{B_{2\lambda R}}\left|D\psi(x)\right|^\frac{np}{n-2\alpha}dx+\left(\int_{B_{2\lambda R}}\left|Du(x)\right|^pdx\right)^\frac{n}{n-2\alpha}\right].
\end{align}
 
 By Lemma \ref{lemmapreliminare}, the hypothesis $V_p(D\psi)\in B^{\gamma}_{2, \infty, \mathrm{loc}}(\Omega)$ implies that $D\psi\in B^{\gamma}_{p, \infty, \mathrm{loc}}(\Omega)$, and since $0<\alpha<\gamma<1$, by Lemma \ref{EP2.3}, $V_p(D\psi)\in B^{\alpha}_{2, \infty, \mathrm{loc}}(\Omega)$ and $D\psi\in B^{\alpha}_{p, \infty, \mathrm{loc}}(\Omega)$.\\
 
 So we can take the supremum for $h\in B_{\frac{R}{4}}(0)$ at the both sides of \eqref{full_c'''}, thus getting
 
 \begin{align}\label{full_c*}
 &\left[V_p\left(Du\right)\right]_{\dot{B}^{\alpha\beta}_{2,\infty}\left(B_{\frac{R}{2}}\right)}\le C\left[V_p\left(D\psi\right)\right]_{\dot{B}^\alpha_{2,\infty}\left(B_{R}\right)}\notag\\
 &\quad+C\left[1+\int_{B_{2R}}g^{\frac{n}{\alpha}}(x)dx+\int_{B_{2\lambda R}}\left|D\psi(x)\right|^\frac{np}{n-2\alpha}dx+\left(\int_{B_{2\lambda R}}\left|Du(x)\right|^pdx\right)^\frac{n}{n-2\alpha}\right]^\sigma,
 \end{align}
 
 where the exponent $\sigma>0$ depends on $n, p$ and $\alpha$ and the constant $c>0$ depends on $n, p, \alpha, \nu, L$, and $R$.\\
 Recalling the definition of the norm in Besov-Lipschitz spaces and using Lemma \ref{EP2.3}, we have 
 
 \begin{align}\label{conclusion3}
 &\left[V_p\left(Du\right)\right]_{\dot{B}^{\alpha\beta}_{2,\infty}\left(B_{\frac{R}{2}}\right)}\le C\left\Arrowvert V_p\left(D\psi\right)\right\Arrowvert_{B^\gamma_{2,\infty}\left(B_{R}\right)}\notag\\
 &\quad+C\left[1+\int_{B_{2R}}g^{\frac{n}{\alpha}}(x)dx+ \int_{B_{2\lambda R}}\left|D\psi(x)\right|^\frac{np}{n-2\alpha}dx+\left(\int_{B_{2\lambda R}}\left|Du(x)\right|^pdx\right)^\frac{n}{n-2\alpha}\right]^\sigma.
 \end{align}
 
Recalling that, for $0<\alpha<\gamma<1$, we have $p<\frac{np}{n-2\alpha}<\frac{np}{n-2\gamma}$, we get 

\begin{align}\label{estimate3pf*}
\left[V_p\left(Du\right)\right]_{\dot{B}^{\alpha\beta}_{2,\infty}\left(B_{\frac{R}{2}}\right)}\le& C\left(1+\left\Arrowvert Du\right\Arrowvert_{L^p(B_{2\lambda R})}+\left\Arrowvert D\psi\right\Arrowvert_{L^\frac{np}{n-2\gamma}(B_{2\lambda R})}\right.\notag\\
&\left.\quad+\left\Arrowvert V_p\left(D\psi\right)\right\Arrowvert_{B^\gamma_{2,\infty}(B_{R})}+\left\Arrowvert g\right\Arrowvert_{L^\frac{n}{\alpha}\left(B_{2R}\right)}\right)^\sigma,
\end{align}

and applying Lemma \ref{Lemma2.9EP} to the function $V_p\left(D\psi\right)$, we get

\begin{align}\label{estimate3pf}
\left[V_p\left(Du\right)\right]_{\dot{B}^{\alpha\beta}_{2,\infty}\left(B_{\frac{R}{2}}\right)}\le& C\left(1+\left\Arrowvert Du\right\Arrowvert_{L^p(B_{4R})}+\left\Arrowvert V_p\left(D\psi\right)\right\Arrowvert_{B^\gamma_{2,\infty}(B_{4R})}\right.\notag\\
&\left.\quad+\left\Arrowvert g\right\Arrowvert_{L^\frac{n}{\alpha}\left(B_{2R}\right)}\right)^\sigma,
\end{align}

that is \eqref{estimate3}.
\end{proof}

\end{document}